\documentclass[onefignum,onetabnum]{siamart251216}

\usepackage{amsfonts}
\usepackage{graphicx}
\usepackage{epstopdf}
\usepackage{algorithmic}
\ifpdf
  \DeclareGraphicsExtensions{.eps,.pdf,.png,.jpg}
\else
  \DeclareGraphicsExtensions{.eps}
\fi

\newsiamremark{remark}{Remark}
\newsiamremark{hypothesis}{Hypothesis}
\crefname{hypothesis}{Hypothesis}{Hypotheses}
\newsiamthm{claim}{Claim}

\headers{An efficient method for quadrilateral mesh generation}{J. Dai, Z. Qiao, and D. Wang}

\title{An efficient and stable diffusion generated method for quadrilateral mesh generation in general domains\thanks{
The second author was partially supported by NSFC/RGC Joint Research Scheme grant
N\_PolyU5145/24 and Hong Kong Research Grants Council GRF grant 15305624. The third author was partially supported by National Natural Science Foundation of China (Grant No. 12422116), Guangdong Basic and Applied Basic Research Foundation (Grant No. 2023A1515012199), Shenzhen Science and Technology Innovation Program (Grant No. JCYJ20220530143803007, RCYX20221008092843046), Guangdong Provincial Key Laboratory of Mathematical Foundations for Artificial Intelligence (2023B1212010001), and Hetao Shenzhen-Hong Kong Science and Technology Innovation Cooperation Zone Project (No. HZQSWS-KCCYB-2024016). }}

\author{
 Jingwen Dai\thanks{Department of Applied Mathematics, The Hong Kong Polytechnic University, Hung Hom, Kowloon, Hong Kong (\email{jing-wen.dai@connect.polyu.hk}).}
  \and
Zhonghua Qiao\thanks{Department of Applied Mathematics, The Hong Kong Polytechnic University, Hung Hom, Kowloon, Hong Kong (\email{zqiao@polyu.edu.hk}).}
  \and
Dong Wang\thanks{Corresponding author. School of Science and Engineering, The Chinese University of Hong Kong, Shenzhen, Guangdong, China;
  Shenzhen International Center for Industrial and Applied Mathematics, Shenzhen Research Institute of Big Data, Guangdong, China (\email{wangdong@cuhk.edu.cn}).}
}

\usepackage{amsopn}

\makeatletter
\newcommand*{\addFileDependency}[1]{
  \typeout{(#1)}
  \@addtofilelist{#1}
  \IfFileExists{#1}{}{\typeout{No file #1.}}
}
\makeatother

\ifpdf
\hypersetup{
  pdftitle={A fast diffusion generated computational method for quadrilateral meshes},
}
\fi

\usepackage{multirow}
\usepackage{caption}
\usepackage{subcaption}
\usepackage[a4paper,margin=1in]{geometry}

\crefname{figure}{Fig.}{Figs.}   
\Crefname{figure}{Fig.}{Figs.}

\begin{document}

\maketitle

\begin{abstract}
This paper introduces a novel, robust, and computationally efficient framework for high-quality quadrilateral mesh generation on general two-dimensional domains. The core of the proposed approach is a novel method for computing cross fields by minimizing a modified and relaxed Ginzburg--Landau-type energy functional. A key innovation is the extension of the problem from the original, potentially complex domain to a larger regular computational domain. This extension transforms the central computational procedure into an iterative scheme that requires only two straightforward and efficient operations: linear diffusion solved globally via the Fast Fourier Transform (FFT) and point-wise normalization. Notably, our method eliminates the conventional need for generating an intermediate triangular mesh or solving complex nonlinear optimization problems on the irregular domain. We provide a rigorous theoretical analysis, proving that the proposed iterative algorithm guarantees unconditional monotonic decay of the objective functional. Comprehensive numerical experiments demonstrate the method's robustness across a wide range of complex geometries, its significant computational efficiency afforded by the FFT-based diffusion, and its consistent generation of high-quality quadrilateral meshes. This work presents a reliable and theoretically sound alternative to existing mesh generation techniques, with strong potential for practical applications in scientific computing.
\end{abstract}



\section{Introduction}
\label{sec:intro}
Block-structured quadrilateral meshes are fundamental to scientific computing, offering a strategic compromise between the efficiency of structured grids and the geometric adaptability needed for complex domains. The locally structured topology of these meshes is particularly advantageous for numerical methods; it facilitates the implementation of high-order tensor-product schemes and accelerates solver convergence. Moreover, the inherent regularity optimizes memory access patterns, making these meshes highly efficient for large-scale data handling. Consequently, they are extensively employed in diverse applications, ranging from computational fluid dynamics \cite{durst1996parallel,raddatz2005block,zou2024finite} to computer graphics tasks such as texture generation \cite{bokhovkin2023mesh2tex,li2021interactive} and animation~\cite{marcias2013animation}.

In the automatic generation of block-structured quadrilateral meshes, the design of the topological structure—commonly referred to as block decomposition or quad layout generation—remains one of the most challenging components \cite{Armstrong2015,bommes2013quad}. The core idea is to identify a small number of irregular points and construct connecting separatrices (or decomposition lines) that partition the computational domain into multiple blocks, each of which can be mapped to a regular quadrilateral grid. The configuration of irregular points and their connecting separatrices—which effectively define the multi-block structure of the mesh—is strictly governed by the interplay between topological invariants (specifically, the Euler characteristic) and the geometric constraints of the domain boundary. Features such as high-curvature regions and sharp corners impose local valence constraints that must be balanced globally to satisfy topological validity, often limiting the available design space for the block decomposition~\cite{BEAUFORT2017219,bommes2013quad,fogg2017simple}. As a result, developing efficient, robust, and geometry-aware algorithms for block decomposition across diverse geometric domains has become a popular direction in structured mesh generation.


A widely adopted approach to quadrilateral mesh generation is based on computing a smooth cross field~\cite{BEAUFORT2017219,kowalski2013pde,Osting2019}, which defines a locally orthogonal set of directions aligned with the desired mesh orientation. A cross field provides continuous directional guidance over the domain, and its singularities correspond to irregular nodes in the final mesh. Once the field is computed, a quadrilateral layout can be constructed by tracing separatrices—stream lines that connect singularities or terminate at the boundary—thereby inducing a block decomposition suitable for quad meshing.

A typical procedure is illustrated in \cref{fig:workflowCF}. The process begins with the prescription of unit normal vectors on the boundary of a closed planar domain $\Omega_1 \subset \mathbb{R}^2$. These boundary conditions are employed to construct a boundary-aligned cross field by computing an associated representation field. Here, boundary alignment implies that one of the cross directions coincides with the unit normal to the boundary. This representation field is typically obtained by approximately solving the limiting Ginzburg--Landau model~\cite{bethuel1994ginzburg}, which involves minimizing a Dirichlet energy functional subject to a unit-norm constraint,
\begin{equation}
\label{eq:originproblem}
\begin{aligned}
    & \min_{\mathbf{u} \in H^1(\Omega_1; \mathbb{R}^2)} 
    && E(\mathbf{u}) := \frac{1}{2} \int_{\Omega_1} |\nabla \mathbf{u}(x)|^2 \, \mathrm{d}x \\
    & \text{subject to} 
    && |\mathbf{u}(x)| = 1 \quad \quad \text{for a.e. } x \in \Omega_1, \\
    &
    && \mathbf{u}(x) = \mathbf{g}(x) \quad \text{on } \partial \Omega_1.
\end{aligned}
\end{equation}
The resulting vector field is post-processed to construct a quadrilateral layout; specifically, the domain is partitioned into subdomains defined by separatrices and boundaries. Subsequently, a structured quadrilateral mesh is mapped into each four-sided block. This paper focuses on the mathematics of cross-field generation and the development of efficient algorithms.
\begin{figure}[t!]
\centering
\includegraphics[scale=0.3]{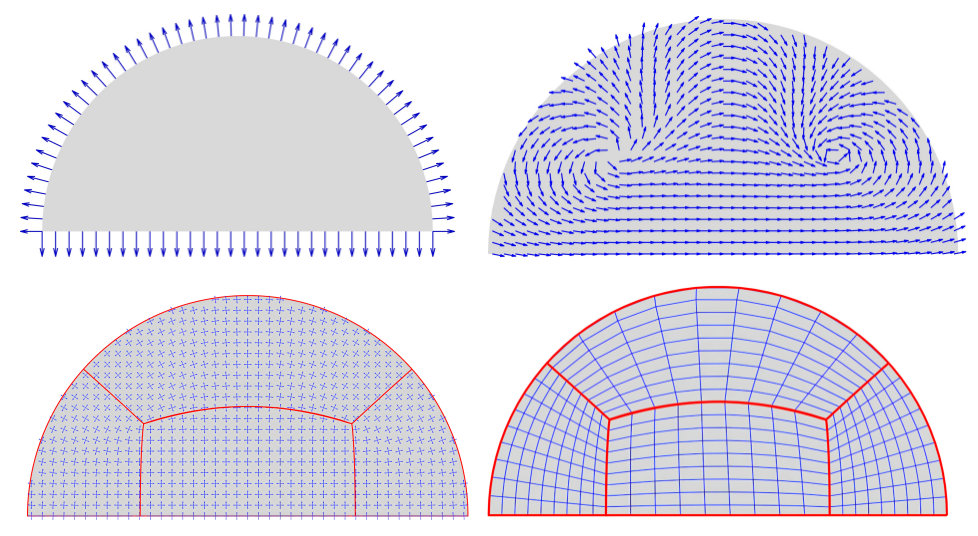}
\caption{Procedure of quadrilateral mesh generation using a cross field. From top left to bottom right: (a) prescribed boundary conditions; 
(b) computed smooth cross field; (c) extracted singularity graph and parameterization; (d) final quadrilateral mesh aligned with the cross field. See Section~\ref{sec:intro}.}
\label{fig:workflowCF}
\end{figure}

Incorporating~\eqref{eq:originproblem} into the vector field design process establishes rigorous theoretical guarantees for quadrilateral mesh generation~\cite{bethuel1994ginzburg,kowalski2013pde,Osting2019}. The solution of~\eqref{eq:originproblem} ensures the smoothness of the solution away from singularities, while naturally allowing for the formulation of isolated singularities. These singularities correspond to irregular nodes in the final mesh and are guaranteed to appear and satisfy Poincar\'e--Hopf theorem. Moreover, the separatrices of the resulting cross field provably partition the domain into a valid quad layout, enabling structured quadrilateral mesh generation. These establish a robust and theoretically grounded framework for cross field construction and quad layout generation.

Subsequent numerical approaches address primarily the challenge posed by the unit-norm constraint, which leads to an empty feasible set due to topological obstructions. Kowalski, Ledoux and Frey~\cite{kowalski2013pde} considered a gradient flow formulation of the Dirichlet energy and employed Lagrange multipliers to enforce the unit-norm constraint linearly. At each iteration, the system solves the following system to a proper time step $\Delta t$ 
\begin{equation*}
\begin{cases}
    \partial_t \mathbf{u} = \Delta \mathbf{u} & \text{in } \Omega_1 \times (0, \Delta t], \\
    \mathbf{u}(x) = \mathbf{g}(x) & \text{on } \partial\Omega_1 \times (0, \Delta t], \\
    \mathbf{u}(x, 0) = \mathbf{u}^{(n)}(x) & \text{in } \Omega_1, \\
    \mathbf{u}(x, \Delta t) \cdot \mathbf{u}^{(n)}(x) = 1 & \text{in } \Omega_1,
\end{cases}
\end{equation*}to obtain $\mathbf{u}^{(n+1)}$. 
Fogg \textit{et al.}~\cite{fogg2013multi} synthesize a reference angular field by leveraging Bunin's theory of anisotropic sizing. To approximate this reference field, they minimize a composite energy functional comprising a Dirichlet smoothness term and two alignment penalty terms. These potentials are designed to penalize deviations from both the values and the gradients of the reference field. The behavior of the resulting field is governed by a set of user-prescribed weights assigned to these three energy terms. The mechanism by which these hyperparameters influence the results has not been fully understood, and certain ratios were found to increase the number of singularities. Bommes, Zimmer, and Kobbelt~\cite{bommes2009mixed} reformulated the problem using the angle representation of cross-field vectors, minimizing the energy functional \( E^{\theta} = \int_{\Omega} |\nabla \theta|^{2} \, dx \) over the quotient space \( \mathbb{R} \big/   \mathbb{Z}_4  \) to encode the unit norm constraint implicitly. In this representation, the angle variable \( \theta \) does not lie in a linear space, so the formulation is not equivalent to solving \( \Delta \theta = 0 \), and the resulting system requires a nonlinear solver with mixed-integer optimization.

Another idea is to consider energy generalizations to the Ginzburg--Landau functional~\cite{du2021maximum}
\begin{equation}
\label{eq:GL functional}
E^{\mathrm{GL}}_{\varepsilon}(\mathbf{u}) = E^{\mathrm{D}} + E^{\mathrm{P}}_{\varepsilon} = \frac{1}{2} \int_{\Omega} \lvert \nabla \mathbf{u} \rvert^{2} \,\mathrm{d}x + \frac{1}{4 \varepsilon^{2}} \int_{\Omega} \bigl( \lvert \mathbf{u} \rvert^{2} - 1 \bigr)^{2} \,\mathrm{d}x.
\end{equation}
In the above form, the unit norm constraints on the vectors are relaxed by adding a penalty term $E^P = \frac{1}{4 \varepsilon^{2}}\int_{\Omega}\left(|\mathbf{u}|^{2}-1\right)^{2} \,  \mathrm{d}x $ to the energy functional. In~\cite{BEAUFORT2017219}, Beaufort \textit{et al.} minimized the Ginzburg--Landau energy by solving the Euler--Lagrange equation directly and obtained a system of nonlinear partial differential equations solved with the Newton method based on Crouzeix--Raviart interpolation in a finite element framework. Jezdimirovi\'{c} \textit{et al.}~\cite{jezdimirovic2019mul} extended this formulation by including the term \( \mathbf{u}|\nabla \mathbf{u}|^{2} \) using the strategy proposed in~\cite{bommes2009mixed} and derived a formulation equivalent to \( \Delta \theta = 0 \). Blanchi \textit{et al.}~\cite{blanchi2021global} combined the Ginzburg--Landau model with the periodic global parameterization method~\cite{ray2006periodic} by introducing a periodic function \( z(p) = e^{2\mathrm{i}\pi\theta(p)} \) and formulating the minimization problems as
\[
z_{\varepsilon} = \operatorname*{arg\,min}_{z \in \mathbb{C}} \left\{ \frac{1}{2} \int_{\Omega} \lvert \nabla z - 2\pi \mathrm{i} \mathbf{V} z \rvert^{2} \,\mathrm{d}x + \frac{1}{4\varepsilon^{2}} \int_{\Omega} \bigl( \lvert z \rvert^{2} - 1 \bigr)^{2} \,\mathrm{d}x \right\}.
\]
to optimize the vector field. However, methods based on the Euler--Lagrange equations often exhibit strong nonlinearity and coupling, which pose significant challenges to numerical stability and efficiency. Viertel and Osting~\cite{Osting2019} proposed solving the Ginzburg--Landau-based cross field model using the diffusion generated method~\cite{Ruuth2001}, which is originally developed for simulating mean curvature flow \cite{esedoglu2015threshold,MBO1993} and extended to find general target-valued harmonic maps \cite{osting2020diffusion,quan2026unconditional,wang2022efficient}. This approach alternates between diffusion (solving a heat equation) and normalization. After each diffusion step, the solution is projected back onto the original constraint manifold (i.e., the unit circle), thereby maintaining the unit-norm condition and effectively approximating minimizers of the relaxed Ginzburg--Landau energy. However, the unconditional stability for the diffusion generated method is not straightforward in particular for general domains.

Many existing methods require pre-generated triangular meshes and involve complex preprocessing. Furthermore, reliance on nonlinear iterations or auxiliary PDE solvers often lacks theoretical stability guarantees. While cross-field methods are promising, energy minimization models such as \eqref{eq:originproblem} are often overly constrained by boundary conditions, leading to inflexible singularity patterns. This geometric sensitivity compromises robustness in complex domains. Consequently, balancing computational efficiency, robustness, and theoretical rigor remains a significant challenge.

In this paper, we propose a novel relaxation of~\eqref{eq:originproblem}. By incorporating the diffuse domain method (DDM)~\cite{guo2021diffuse,li2009solving,yu2020higher}, we reformulate the original partial differential equation problem over a larger and regular rectangular domain $\Omega_2$. This approach introduces a smooth phase-field function to replace the sharp boundary of the original geometry with a $O(\epsilon)$ diffuse transition layer, thereby eliminating the need for unstructured triangular meshes. Furthermore, we adopt the exponential of the Laplacian operator as a diffusion-based generation mechanism to approximate the Dirichlet energy over the extended domain. As a result, we formulate the following relaxed energy and solve a relaxed minimization problem subject to a unit-norm constraint over a regular rectangle domain. 

Based on the relaxed problem, we design an algorithm that is simple to implement, computationally efficient and numerically stable. It involves only two steps: diffusion and normalization. The method applies to general planar domains for computing cross fields and generating quadrilateral meshes. It does not rely on a pre-existing triangular mesh. The convolution between the free-space heat kernel and a function with compact support is computed efficiently with the Fast Fourier Transform (FFT). The method remains effective on irregular domains by smoothly extending them to a slightly larger square region. We also show that the algorithm ensures energy decreases at every iteration without requiring additional conditions.

The remainder of this paper is organized as follows. We begin in Section~\ref{sec:foundation} by outlining the topological foundations of quadrilateral meshing, introducing the cross field representation, and establishing its connection to the Ginzburg--Landau variational model. Section~\ref{sec:DM} details our core computational framework for cross field generation, which leverages a convolution-based diffuse domain method (C-DDM) within a diffusion-generated method. The derivation of a modified, computationally efficient Ginzburg--Landau-type energy and the associated iterative scheme for finding its minimizers are presented in Section~\ref{sec:Derivation}. Section~\ref{sec:Implementation} elaborates on the complete implementation pipeline, from solving for the cross field to the final quadrilateral mesh generation. Comprehensive numerical validation is provided in Section~\ref{sec:Numerical}, demonstrating the method's robustness, adaptability to boundary conditions, controllability over singularity placement, and performance on a composite geometric domain. We conclude in Section~\ref{sec:conclusions} with a summary of our key contributions and a discussion of potential future research directions.

\section{Quad meshes, cross fields, and the Ginzburg--Landau model}
\label{sec:foundation}
This section reviews the topological foundations of quadrilateral meshing, focusing on the constraints imposed by the Euler characteristic. These topological invariants necessitate the presence of irregular vertices and connecting separatrices in any quadrilateral mesh of a general domain. A principled approach to generating such fields is provided by the Ginzburg--Landau model, a variational framework that ensures smoothness and topological alignment. Building upon this foundation, we develop a fast and robust computational framework for cross field generation and subsequent quad mesh construction.

\begin{figure}[t!]
  \centering
  \includegraphics[scale=0.15]{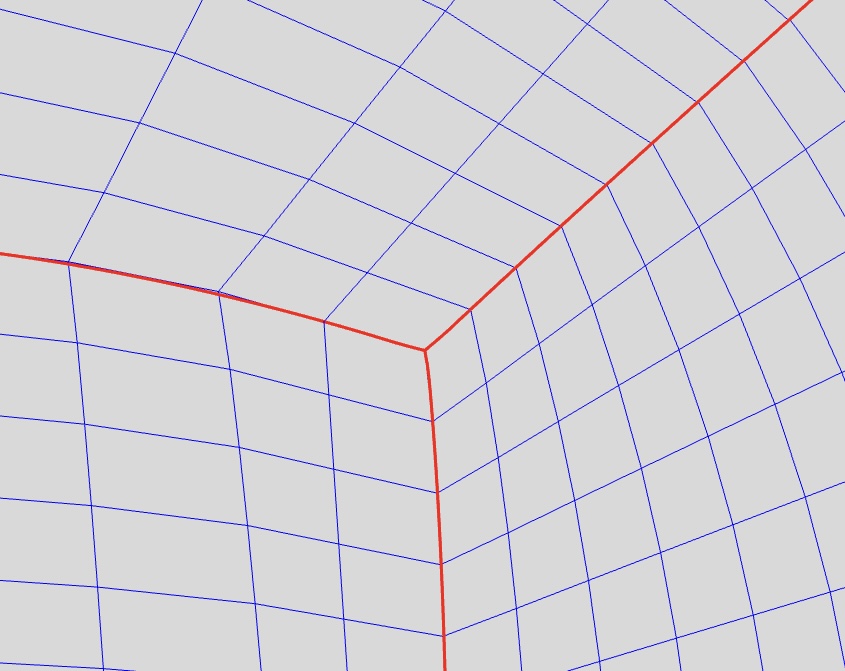}
  \hspace{0.5cm} 
  \includegraphics[scale=0.15]{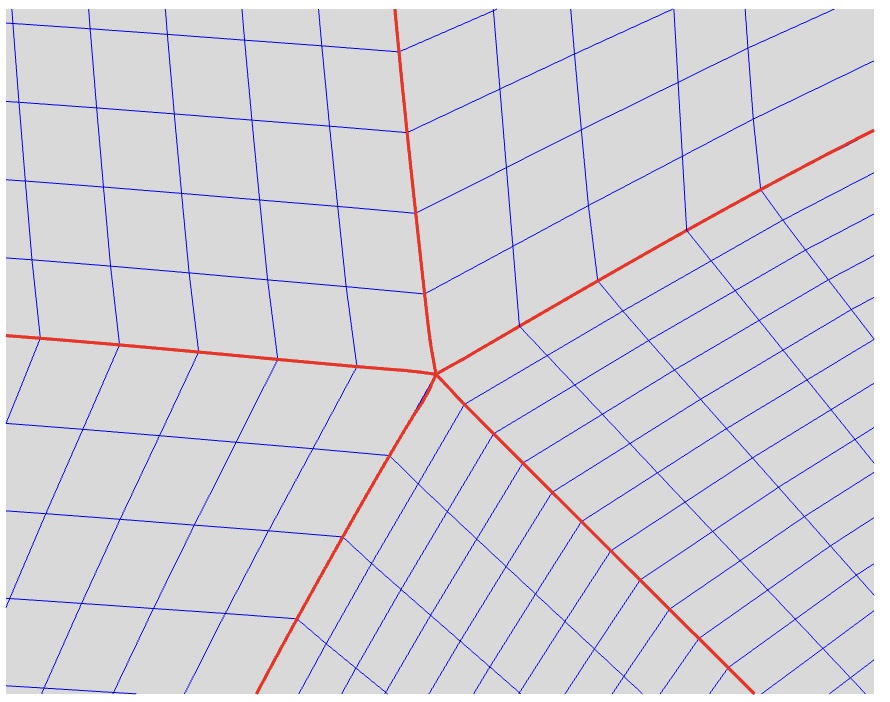}
  \caption{Examples of irregular points in quadrilateral mesh topology. 
(a) A valence-3 singularity, and (b) a valence-5 singularity. 
These irregular points deviate from the regular valence-4 configuration 
and are necessary to satisfy topological constraints in complex domains. See Section~\ref{sec:constraints}.}
  \label{fig:irrgularpoints}
\end{figure}

\subsection{Quadrilateral mesh generation constraints}
\label{sec:constraints}

Let $\mathcal{M} = (\mathcal{V}, \mathcal{E})$ denote a quadrilateral mesh with vertex set $\mathcal{V}$ and edge set $\mathcal{E}$. For any $v \in \mathcal{V}$, the \emph{valence} (or \emph{degree}) $d(v)$ is defined as the number of edges incident to $v$. We partition $\mathcal{V}$ into the set of interior vertices $\mathcal{V}^\circ$ and boundary vertices $\partial\mathcal{V}$. A vertex $v$ is said to be \emph{regular} if its valence matches the ideal configuration of a structured mesh, namely
\begin{equation*}
d(v) = \bar{d}(v) := 
\begin{cases} 
4, & v \in \mathcal{V}^\circ, \\ 
2, & v \in \partial\mathcal{V}.
\end{cases}
\end{equation*}
Vertices where $d(v) \neq \bar{d}(v)$ are referred to as \emph{extraordinary} (or \emph{irregular}) nodes; typical configurations for $d(v) \in \{3, 5\}$ are illustrated in \cref{fig:irrgularpoints}. To quantify the local topological deficit, we define the \emph{vertex index} $\mathcal{I}(v)$ as
\begin{equation*}
\mathcal{I}(v) = \frac{1}{4} \left( \bar{d}(v) - d(v) \right).
\end{equation*}

Furthermore, the indices of irregular vertices must satisfy a global topological constraint governed by the Euler characteristic of the domain, a fundamental topological invariant. In other words, the sum of the indices over all mesh vertices must equal the Euler characteristic \( \chi(D) \) of the domain \( D \) \cite{BEAUFORT2017219,fogg2017simple}, where \( b \) is the number of boundary components,
\[\sum_{v \in \mathcal{V}} \mathcal{I}(v) = \chi(D) = 2 - b.\]
It is worth emphasizing that this index constraint, while a necessary condition for generating a valid quadrilateral mesh, is not sufficient. In practical settings, boundaries with sharp corners and curved domain boundaries may introduce additional irregular vertices in order to preserve mesh quality and minimize distortion~\cite{Blacker1991,fogg2017simple}. 

The concept of cross fields offers a mathematical basis for handling these constraints by representing the target mesh orientation as a continuous field. This approach enables the implicit modeling of irregular vertices, where the streamline of the field guides the quadrilateral alignment. Consequently, this representation enables the implicit detection and computation of singularities associated with irregular vertices.

\begin{figure}[t!]
    \centering
    \begin{subfigure}[b]{0.22\textwidth}
        \includegraphics[width=\linewidth, height=3.1cm, keepaspectratio]{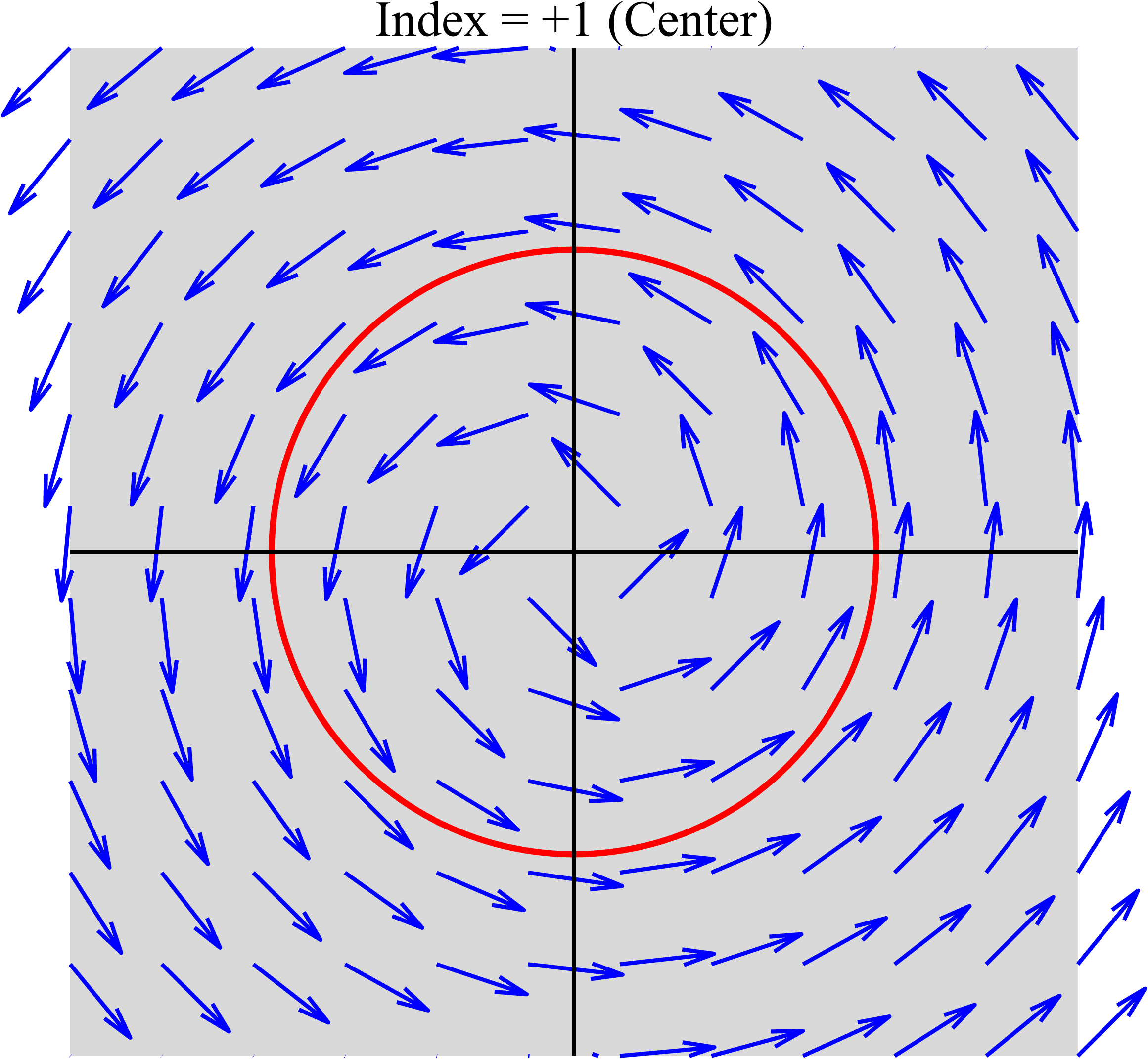}
        \caption{$+1$}
        \label{fig:general_field_index=1}
    \end{subfigure}
    \hspace{0.01\textwidth}
    \begin{subfigure}[b]{0.22\textwidth}
        \includegraphics[width=\linewidth, height=3.5cm, keepaspectratio]{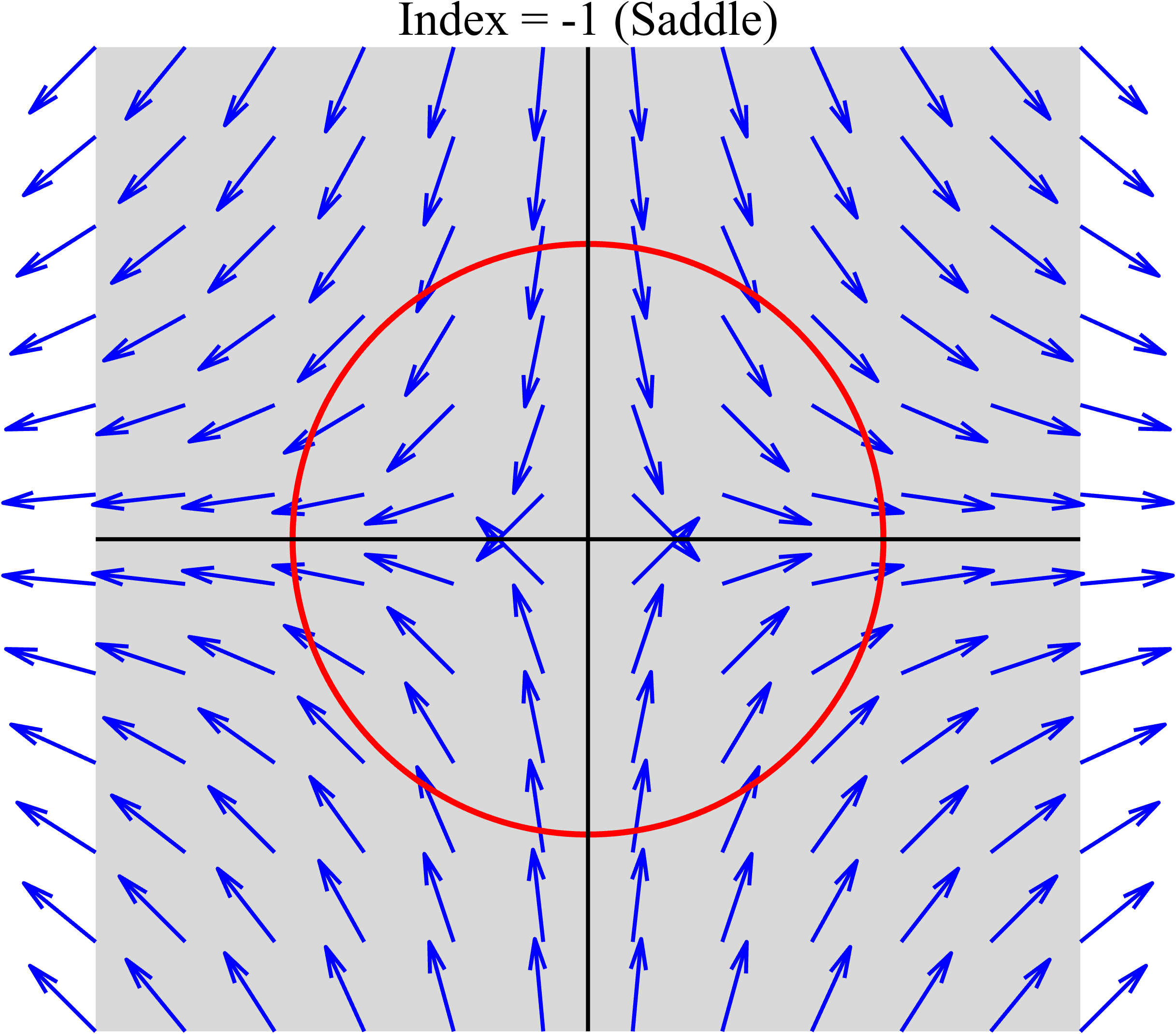}
        \caption{$-1$}
        \label{fig:general_field_index=-1}
    \end{subfigure}
    \hspace{0.01\textwidth}
    \begin{subfigure}[b]{0.22\textwidth}
        \includegraphics[width=\linewidth, height=3.0cm, keepaspectratio]{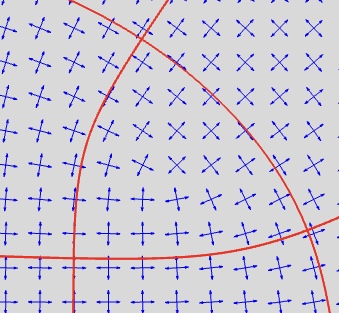}
        \caption{$\frac{1}{4}$}
        \label{fig:corss_fields_index_1/4}
    \end{subfigure}
    \hspace{0.01\textwidth}
    \begin{subfigure}[b]{0.22\textwidth}
        \includegraphics[width=\linewidth, height=0.85\linewidth]{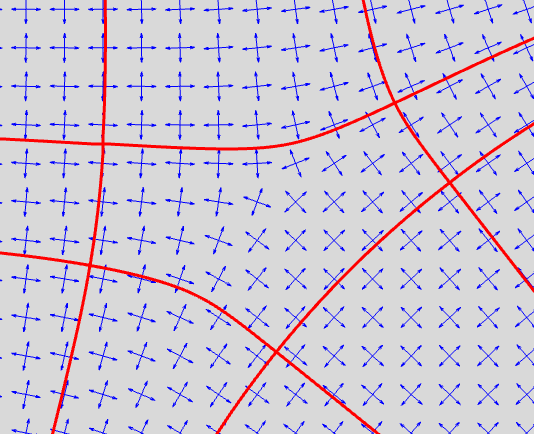}
        \caption{$-\frac{1}{4}$}
        \label{fig:corss_fields_index_-1/4}
    \end{subfigure}
    \caption{Comparison of singularities in vector and cross fields. 
(a) and (b) show classical vector field singularities with index $+1$ (center) and $-1$ (saddle), 
where the field completes a full rotation around the singularity. 
(c) and (d) depict cross field singularities with index $+1/4$ and $-1/4$, 
in which the cross field undergoes a quarter-turn counterclockwise or clockwise, respectively. See Section~\ref{sec:crossfield}.}
    \label{fig:singularities_of_general_and_cross_fields}
\end{figure}

\subsection{Cross fields}

\label{sec:crossfield}
Let $\Omega \subset \mathbb{R}^2$ be a compact planar domain with a piecewise smooth boundary. A \textbf{cross} is defined as an element of the quotient space $\mathbb{S}^1 \big/ \mathbb{Z}_4$. Here, the cyclic group $\mathbb{Z}_4 \cong \{0, 1, 2, 3\}$ acts on the unit circle $\mathbb{S}^1 \subset \mathbb{C}$ via the rotation map $k \cdot z = e^{i k \pi / 2} z$ for $z \in \mathbb{S}^1$ and $k \in \mathbb{Z}_4$. It reflects the \( \frac{\pi}{2} \)-rotational symmetry of the circle. The cross field could be defined as a map:
\[
f : \Omega \rightarrow \mathbb{S}^1 / \mathbb{Z}_4  \ \cup \{\mathbf{0}\}.
\]
Points where the function $f$ maps to zero are called singularities, where the field is non-smooth or ill-defined. Away from finite singularities, $f$ remains continuous and integrable in $\Omega$~\cite{BEAUFORT2017219,jezdimirovic2019mul}. 

To overcome the computational difficulties arising from the rotational symmetry of cross fields, one can introduce the representation field as a mathematical formulation. Let $\theta(p) \in [0, \pi/2)$ denote the principal argument of the cross vector at point $p$, then the representation field is defined to be the map $R: \mathbb{S}^1/\mathbb{Z}_4 \rightarrow \mathbb{S}^1$, 
\[
R(f(p)) = e^{i4\theta(p)} \in \mathbb{C}, \quad \theta(p) \in [0, \frac{\pi}{2}).
\]
Within the representation field, singularities are identified as points where the index, or winding number, is non-zero. Let $\gamma: [0,1] \to \mathbb{R}^2$ be a simple, closed, piecewise $C^1$ curve with $\gamma(0) = \gamma(1)$, and $p \in \text{int}(\gamma)$. The index at point $p$ is defined as the integral of the argument along the contour $\gamma$,
\[
\text{Ind}_{\text{R}}(p) = \frac{1}{2\pi} \oint_{\gamma} d\theta (p).
\]
From the argument mapping $\theta_r(p) = 4 \theta(p)$, the loop integral (index) of the representation field can be extended to define the index of the cross field $\text{Ind}_{f}(p) = \frac{1}{4} \text{Ind}_{\text{R}}(p)$. The cross field indices coincide with the definition of irregular mesh vertices in Section~\ref{sec:constraints}; via the Poincar\'e--Hopf theorem, this ensures that the configuration of cross field singularities satisfies global topological constraints of quad mesh. A visual comparison of the singularities in cross fields and vector fields is illustrated in~\Cref{fig:singularities_of_general_and_cross_fields}. 

Based on the cross-field, we define the cross-field streamlines, which serve as another mathematical foundation for domain decomposition. Formally, a streamline $\gamma(t)$ is an integral curve that satisfies the following constraint:
\begin{equation}
    \frac{d\gamma}{dt} \in \left\{ R_k(\mathbf{f}(\gamma(t))) \mid k \in \{0, 1, 2, 3\} \right\},
    \label{eq:streamline_inclusion}
\end{equation}
where $\mathbf{f}(\cdot)$ represents the cross field and $R_k$ maps the field to one of its four symmetric branches. To trace a specific trajectory, we determine the branch index $k$ locally to enforce continuity by minimizing the deviation from the current propagation direction
\begin{equation}
\label{eq:streamline condition}
    \begin{aligned}
    k &= \arg \min_{j \in \{0,1,2,3\}} 
\ d_{\text{cross}}\left( R_j(\theta(\gamma(t))),\ \arg(\dot{\gamma}(t)) \right),
\end{aligned}
\end{equation}
where the distance function $d_{\text{cross}}$ is \[d_{\text{cross}}(\theta_1, \theta_2) = \min_{m \in \mathbb{Z}} \left| \theta_1 - \theta_2 - m \frac{\pi}{2} \right|.\]


\subsection{The Ginzburg--Landau model for cross fields}
\label{sec:Ginzburg}
To compute the representation field described in Section~\ref{sec:crossfield}, a widely adopted approach is to identify the field with a vector function $\mathbf{u} \in L^2(\Omega_1;\mathbb{R}^2)$ and employ the Ginzburg--Landau model. Originating from the theory of superconductivity~\cite{bethuel1994ginzburg,Ginzburg2009}, this classical model has been extensively applied to the design of direction fields with topological singularities~\cite{BEAUFORT2017219,kowalski2013pde,Osting2019}. Mathematically, the Ginzburg--Landau energy arises as a relaxation of \eqref{eq:originproblem}, where the strict unit-length constraint is replaced by a penalization term. This relaxed formulation not only facilitates the analysis of solution regularity but also accommodates the existence of singularities, thereby enabling the optimization of their spatial distribution. Consequently, the Ginzburg--Landau framework serves as the theoretical foundation and primary motivation for the method proposed in this work.

As established in \cite{bethuel1994ginzburg,Osting2019}, for a boundary condition 
\(\mathbf{g} : \partial \Omega \to \mathbb{S}^1\) with non-zero degree, the minimizers 
of the Ginzburg--Landau energy converge (as \(\varepsilon \to 0\)) to a canonical harmonic 
map \(\mathbf{u}_\star\). This limit map is smooth on 
\(\Omega \setminus \{ a_i \}_{i=1}^N\) and has a finite set of topological singularities 
\(\{ a_i \}\). Crucially, the limiting harmonic map $\mathbf{u}_\star$ exhibits a fundamental property: its Dirichelt energy is predominantly concentrated in the neighborhood of the singularities. This phenomenon implies that the strategic placement of irregular vertices (singularities) is primary, while the field remains smooth and uniform elsewhere. The asymptotic analysis for $\varepsilon$ further demonstrates how \eqref{eq:originproblem} determines both the locations and indices of these singularities. In this limit, the energy expansion takes the following asymptotic form:
\begin{equation*}
E^{\mathrm{GL}}_{\varepsilon}(\mathbf{u}_\varepsilon) = \pi \sum_{i=1}^N d_i^2 \log\left(\frac{1}{\varepsilon}\right) + W(\{a_i\}) + o(1) \quad \text{as }\, \varepsilon \rightarrow 0,
\end{equation*}
where \(d_i \in \mathbb{Z}\) denotes the index of the singularity at \(a_i\) and \(W(\{a_i\})\) is the renormalized energy. In this asymptotic expansion, the first term determines the indices of the singularities while the second term governs their locations. According to the Ginzburg--Landau theory, only low-order singularities with index $\pm\frac{1}{4}$ are produced. Such low-index singularities correspond to mild irregularities in the resulting mesh and help to minimize distortion, making them desirable for high-quality quadrilateral meshing. The renormalized energy $W(\{a_i\})$ imposes two key constraints on the location of singularities: (i)~singularities should not be placed too close to each other; and (ii)~singularities should stay away from the boundary.

In summary, while \eqref{eq:originproblem} provides a theoretically elegant framework for constraining singularity configurations, its practical application in geometric settings requires careful examination. Beyond the computational challenges discussed in Section~\ref{sec:crossfield}, a core limitation lies in the model's renormalized energy, which lacks an intrinsic mechanism to effectively balance the competing penalties (i) and (ii). An imbalance in this trade-off can lead to distorted mesh elements with suboptimal aspect ratios. In addition, practical application requires stable parameter choices that provide designers with intuitive control over the singularity layout. Another critical observation is that imposing a Dirichlet boundary condition is often overly restrictive for mesh generation, since strict boundary constraints are not essential for determining interior singularity locations and separatrix structures. To address these limitations, we introduce a new model in the following section.

\section{The proposed modified Ginzburg--Landau energy}\label{sec:DM}
This section presents the derivation of a modified Ginzburg--Landau energy model for computing cross fields. Our approach begins with a computationally efficient refinement of the traditional diffuse domain method. By combining this with the exponential of the Laplacian operator, we approximate and reformulate the original problem \eqref{eq:originproblem} on an extended rectangular domain $\Omega_2$. This reformulation yields a modified Ginzburg--Landau model with favorable mathematical properties. Specifically, we prove the existence of a solution with the unit-norm constraint in a well-defined space. These theoretical results provide the foundation for the stable iterative numerical methods introduced subsequently.

\subsection{Convolution-based diffuse domain method}\label{sec:CDDM}

We introduce a convolution-based diffuse domain method as an effective and robust approach that enables fast computation while preserving stability and adaptability to complex geometries. The diffuse domain method (DDM) is a numerical technique for solving partial differential equations in domains with general geometries \cite{guo2021diffuse,li2009solving}. By extending the governing equations to a regular computational domain, the use of body-fit meshes is avoided. The reformulated equation incorporates a lower-order term scaled by a smooth function $\phi(x)$, which varies according to different boundary conditions.

The generation of cross fields is formulated as a constrained minimization problem of the Dirichlet energy. Ideally, we seek a representation field $\mathbf{u} \in H^1(\Omega_1; \mathbb{R}^2)$ that approximates the solution of~\eqref{eq:originproblem}. To address the numerical difficulties imposed by the pointwise unit-norm constraint, we employ an operator splitting strategy. By initially relaxing the constraint, the problem reduces to the unconstrained minimization of the Dirichlet energy. The associated Euler-Lagrange equation is the Laplace equation, $\Delta \mathbf{u} = 0$, subject to the prescribed Dirichlet boundary conditions.

Subsequently, to accommodate the geometric complexity of $\Omega_1$ and efficiently enforce Dirichlet boundary conditions on irregular interfaces, we adopt the diffuse domain method. Let $\Omega_1$ be a bounded domain embedded within a larger, regular computational domain $\Omega_2$ (See \cref{fig:diagram} for example). To facilitate the solution of the governing equation on the extended domain $\Omega_2$, the following formulation has proven to be both effective and robust \cite{li2009solving,yu2020higher}:
\begin{equation}
\label{eq:approximaton2}
    \Delta \mathbf{u} - \epsilon^{-3} B(\phi)(\mathbf{u} - \mathbf{g}) = \mathbf{0}.
\end{equation}
Here, $\phi$ is a phase-field function that provides a smooth approximation of the characteristic function $\chi_{\Omega_1}$ for the domain $\Omega_1$. The term $B(\phi)$, typically defined as $\phi^2(1-\phi)^2$, approximates the boundary Dirac delta function $\delta_{\partial \Omega_1}$ concentrated on the boundary. Furthermore, $\mathbf{g}$ represents an extension of the boundary data $\mathbf{u}|_{\partial \Omega_1}$ to a tubular neighborhood of $\partial \Omega_1$.

A commonly used construction of the corresponding phase-field function $\phi$ is based on the signed distance function $r(x)$ to the region $\Omega_1$, which possesses favorable analytical properties \cite{Abels2015,burger2017analysis},
\begin{equation*}
\label{eq:tanhphasefunction}
    \phi(x) := \frac{1}{2} \left( 1 - \tanh\left(\frac{3r(x)}{\varepsilon}\right) \right), \quad x \in \Omega_2.
\end{equation*} 
However, the construction of such a signed distance function is highly sensitive to boundary regularity \cite{burger2017analysis} and  nontrivial for complex geometries~\cite{peng1999pde,sussman1999efficient}. 
The phase-field function generated via the convolution approach shares essential properties with those constructed from signed distance functions, such as smoothness, monotonicity across the interface, and a controllable transition width. A standard approach involves smoothing the characteristic function of the domain using a Gaussian kernel, defined as
\[
G_\varepsilon(x) = \frac{1}{(2\pi \varepsilon^2)^{d/2}} \exp\left( -\frac{|x|^2}{2\varepsilon^2} \right), \quad \phi(x) = \left( G_\varepsilon * \chi_{\Omega_1} \right)(x) = \int_{\mathbb{R}^d} G_\varepsilon(x - y) \, \chi_{\Omega_1}(y) \, \mathrm{d}y,
\]
where $\varepsilon > 0$ determines the width of the transition layer and $d$ denotes the spatial dimension. Compared to signed distance formulations, this convolution-based construction offers significant computational advantages: it is straightforward to implement, avoids the need for solving auxiliary PDEs or optimization problems, and is naturally compatible with structured grids. For discussions on more robust or higher-order smoothing techniques, we refer the reader to~\cite{Geusebroek2003,he2012guided}.
\begin{figure}[t!]
  \centering
  \label{fig:DDM}\includegraphics[scale=0.3]{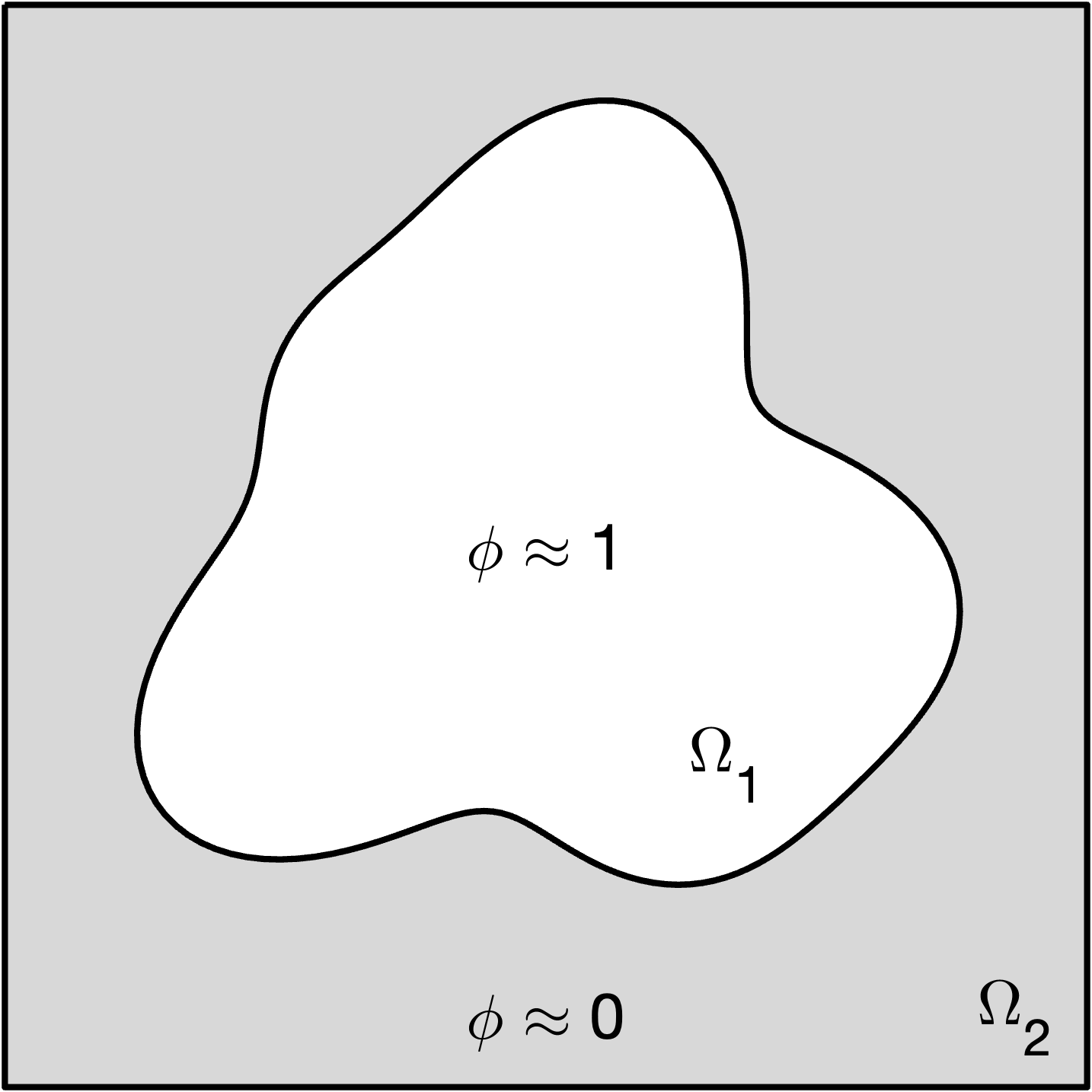}
  \caption{Domain $\Omega_1 \subset \Omega_2 \subset \mathbb{R}^2$ described by a phase-field function $\phi$. See Section~\ref{sec:CDDM}.} \label{fig:diagram}
\end{figure} 

\subsection{Modified energy functional with exponential of the Laplacian}
For the original formulation, the solution corresponds to the minimizer of the following generalized energy functional:
\begin{equation*}    
 \mathbf{u} = \arg\min_{\mathbf{v} \in \mathcal{A}} \; E(\mathbf{v}),
\end{equation*}
where $E(\mathbf{v})$ is the Dirichlet energy, 
\begin{equation*}
    E(\mathbf{v}) = \frac{1}{2} \int_{\Omega_1} |\nabla \mathbf{v}|^2 \, \mathrm{d}x
\end{equation*}
and 
\[\mathcal{A} = \left\{ \mathbf{v} \in H^1(\Omega_1; \mathbb{R}^2) \, \, | \quad \mathbf{v}|_{\partial \Omega_1} = \mathbf{g}, \quad |\mathbf{v}(x)| = 1 \, \, \text{ for \ a.e. } \, \, x \in \Omega_1 \right\} . \]
In the case of \eqref{eq:approximaton2} within the diffuse domain method, this can be interpreted as augmenting the Dirichelt energy with a penalty term, resulting in the following modified energy functional,
\begin{equation}
\label{eq:modiproblem}
E(\mathbf{u}) = \frac{1}{2} \int_{\Omega_2} |\nabla \mathbf{u}(x)|^2 \, \mathrm{d}x 
+ \frac{1}{2} \int_{\Omega_2} \varepsilon^{-3} B(\phi(x))\, |\mathbf{u}(x) - \mathbf{g}(x)|^2 \, \mathrm{d}x.
\end{equation}
When the domain $\Omega_2$ is chosen to be sufficiently large, the influence of its boundary conditions becomes negligible. Therefore, we impose the periodic boundary conditions , leading to the following identity by Green's formula,
\begin{equation}
\label{eq:inte}
\int_{\Omega_2} |\nabla \mathbf{u}|^2 \, \mathrm{d}x = \int_{\Omega_2} - \langle \mathbf{u}, \Delta \mathbf{u} \rangle \, \mathrm{d}x,
\end{equation}
where $\langle \cdot, \cdot \rangle$ denotes the pointwise Euclidean inner product. 
For $\tau > 0$, $e^{\frac{\tau}{2}\Delta} \mathbf{u}$ denotes the solution at $t = \tau/2$ of the following free space heat diffusion equation:
\begin{equation*}
\begin{cases}
\partial_t \mathbf{u} = \Delta \mathbf{u}, \\
\mathbf{u}(x, t=0) = \mathbf{u}(x),
\end{cases}
\end{equation*}
which can also explicitly be written by $\mathbf{u}(x, \tau/2) = G_{\tau/2} * \mathbf{u}$. Regarding the Laplace operator, we recall the following asymptotic expansion for the heat semigroup \cite{wang2022efficient}:
\begin{equation}
\label{eq:diffusion}
e^{\frac{\tau}{2} \Delta} \mathbf{u} = \mathbf{u} + \frac{\tau}{2} \Delta \mathbf{u} + o(\tau), \quad \text{as } \tau \to 0
\end{equation}
and thus 
\begin{align*}
\langle e^{\frac{\tau}{2} \Delta} \mathbf{u}, e^{\frac{\tau}{2} \Delta} \mathbf{u} \rangle & = \|\mathbf{u}\|^2 + \tau \langle \mathbf{u}, \Delta \mathbf{u} \rangle + o(\tau)   \nonumber \\
& = 1 + \tau \langle \mathbf{u}, \Delta \mathbf{u} \rangle + o(\tau)  \quad \text{as } \tau \to 0.
\end{align*}
Substituting \eqref{eq:inte} and \eqref{eq:diffusion} into \eqref{eq:modiproblem} and dropping some terms independent of $\mathbf{u}$, we obtain an approximation energy:
\begin{equation*}
E_{\tau,\epsilon}(\mathbf{u}) = \int_{\Omega_2} 
- \frac{1}{2\tau} \langle e^{\frac{\tau}{2} \Delta} \mathbf{u}, e^{\frac{\tau}{2} \Delta} \mathbf{u} \rangle 
- \frac{1}{\varepsilon^3} B(\phi) \langle \mathbf{u}, \mathbf{g} \rangle \, \mathrm{d}x.
\end{equation*}
It's straightforward that we have the following facts of the exponential of the Laplacian operator:
\begin{subequations} \label{eq:operator_properties}
    \begin{align}
        \|e^{\frac{\tau}{2} \Delta} \mathbf{u}\|_{L^2(\Omega)} &\leq \|\mathbf{u}\|_{L^2(\Omega)}, \label{eq:contraction} \\
        \langle e^{\tau \Delta} \mathbf{u}, \, \mathbf{v} \rangle &= \langle \mathbf{u}, \, e^{\tau \Delta} \mathbf{v} \rangle, \label{eq:symmetry}
    \end{align}
\end{subequations}
for $\mathbf{u}~,\mathbf{v} \in L_2(\Omega;\mathbb{R}^2)$, where $\Omega$ is a square region. Based on these facts, we have the following lemma concerning the energy functional $E_{\tau,\epsilon}(\mathbf{u})$.

\begin{lemma}\label{lemma:energy}
Let $\Omega_2 \subset \mathbb{R}^2$ be a open, connect and bounded domain with Lipschwz boundary, and let $\mathbf{u} \in L^2(\Omega_2; \mathbb{R}^2)$. Assume that $e^{\tau \Delta}$ is the heat semigroup operator with $\tau > 0$, and that $B(\phi) \in L^\infty(\Omega_2)$, $\mathbf{g} \in L^\infty(\Omega_2; \mathbb{R}^2)$. Then the energy functional $E_{\tau,\varepsilon}(\mathbf{u})$
satisfies the following properties:
\begin{enumerate}
    \item \label{item:boundedness} For fixed $\phi$ and $\varepsilon > 0$, $E_{\tau,\varepsilon}(\mathbf{u})$ is bounded from below on the admissible set 
    \begin{equation*}
        \mathcal{C} := \{ \mathbf{u} \in L^2(\Omega_2; \mathbb{R}^2) \mid \quad | \mathbf{u}(x)| = 1 \text{ for a.e. } x \in \Omega_2 \}.
    \end{equation*} 
    
    \item \label{item:continuity} $E_{\tau,\varepsilon}(\mathbf{u})$ is continuous with respect to $\mathbf{u} \in L^2(\Omega_2; \mathbb{R}^2)$.
    
    \item \label{item:concavity} $E_{\tau,\varepsilon}(\mathbf{u})$ is concave with respect to $\mathbf{u} \in L^2(\Omega_2; \mathbb{R}^2)$.
    
    \item \label{item:derivative} The Fréchet derivative of $E_{\tau,\varepsilon}$ at $\mathbf{u}$ in the direction $\mathbf{v} \in L^2(\Omega_2; \mathbb{R}^2)$ is given by
    \begin{equation*}
\int_{\Omega_2} \left\langle \dfrac{\delta E_{\tau,\epsilon}(\mathbf{u})}{\delta \mathbf{u} }, \mathbf{v} \right\rangle \mathrm{d} x  = \int_{\Omega_{2}} \left\langle \mathbf{v}, -\frac{1}{\tau} e^{\tau \Delta} \mathbf{u} - \frac{1}{\varepsilon^{3}} B(\phi)\mathbf{g} \right\rangle \, \mathrm{d} x.
    \end{equation*}
\end{enumerate}

\end{lemma}

\begin{proof}
    1. From the definition, $B(\phi) = \phi^2(1-\phi)^2 \leq1 $ and $\mathbf{g}$ is the extension of unit norm vectors from boundary to the computational domain. With proper extension mentioned in Section \ref{sec:CDDM}, we have $|\mathbf{g}|\leq 1$ for a.e. $x\in \Omega_2$. Therefore, we have 
 \[
     \int_{\Omega_2} \frac{B(\phi)}{\varepsilon^3} \langle \mathbf{u}, \mathbf{g} \rangle \, \mathrm{d}x 
     \leq \int_{\Omega_2} \frac{B(\phi)}{\varepsilon^3}  \, \mathrm{d}x 
    \leq \frac{|\Omega_2|}{\varepsilon^3}.
    \]
In addition, as for the first term, using \eqref{eq:contraction}, we have
    \[
    \int_{\Omega_2} \langle e^{\tau \Delta} \mathbf{u}, e^{\tau \Delta} \mathbf{u}\rangle \, \mathrm{d}x 
    \leq \int_{\Omega_2} |\mathbf{u}|^2 \, \mathrm{d}x 
    \leq |\Omega_2|.
    \]
   Thus $E_{\tau,\epsilon}(\mathbf{u})$ is bounded form below,
    \[
    E_{\tau,\epsilon}(\mathbf{u}) \geq -(\frac{1}{\tau}+\frac{1}{\varepsilon^3})|\Omega_2|.
    \]
  
    2. Consider $\textbf{u},\textbf{v} \in L^2(\Omega_2; \mathbb{R}^2)$ and \eqref{eq:symmetry}, we have the following fact: 
\begin{align*}
\left| E_{\tau, \varepsilon}(\mathbf{u}) - E_{\tau, \varepsilon}(\mathbf{v}) \right|
&\leq \frac{1}{\tau} \int_{\Omega_2} \left| \left\langle \mathbf{u} + \mathbf{v}, e^{\tau \Delta} (\mathbf{u} - \mathbf{v}) \right\rangle \right| 
+ \frac{1}{\varepsilon^3} B(\phi) \left| \left\langle \mathbf{u} - \mathbf{v}, \mathbf{g} \right\rangle \right| \, \mathrm{d}x \\
&\leq \left( \frac{1}{\tau} + \frac{|\Omega_2|}{\varepsilon^3} \right) \left\| \mathbf{u} - \mathbf{v} \right\|_{L^2(\Omega_2;\mathbb{R}^2)}.
\end{align*}

3. We define the energy functional as
\begin{equation*}
\label{eq:energy_functional}
E_{\tau,\epsilon}(\mathbf{u}) = -f(\mathbf{u}) - \frac{2B(\phi)}{\varepsilon^3} \langle \mathbf{u}, \mathbf{g} \rangle,
\end{equation*}
where \( f(\mathbf{u}) = \int_{\Omega_2} \langle e^{\tau \Delta}\mathbf{u}, e^{\tau \Delta} \mathbf{u} \rangle \, \mathrm{d}x \) is a convex functional on \( L^2(\Omega_2;\mathbb{R}^2) \), and \( \frac{B(\phi)}{\varepsilon^3} \langle \mathbf{u}, \mathbf{g} \rangle \) is linear in \(\mathbf{u}\). Therefore, the energy $E_{\tau,\varepsilon}$ is concave in \( L^2(\Omega_2;\mathbb{R}^2) \).

4. Direct computation yields
\begin{align*}
\int_{\Omega_2} \left\langle \dfrac{\delta E_{\tau,\epsilon}(\mathbf{u})}{\delta \mathbf{u} }, \mathbf{v} \right\rangle \mathrm{d} x
&= \lim_{a \to 0} \frac{E_{\tau,\epsilon}( \mathbf{u} + a \mathbf{v}) - E_{\tau,\epsilon}(\mathbf{u})}{a} \\
&= \int_{\Omega_2}  -\frac{1}{\tau} \langle \mathbf{v}, e^{\tau \Delta} \mathbf{u} \rangle 
- \frac{1}{\varepsilon^3} B(\phi) \langle \mathbf{v}, \mathbf{g} \rangle \, \mathrm{d} x.
\end{align*}

\end{proof}

Due to the supporting property of concave functionals, we have the following inequality,
\begin{equation*}
E_{\tau,\varepsilon}(\mathbf{v}) \leq E_{\tau,\varepsilon}(\mathbf{u}) 
+ \left\langle \frac{\delta E_{\tau,\varepsilon}(\mathbf{u})}{\delta \mathbf{u}}, \mathbf{v} - \mathbf{u} \right\rangle, 
\quad \forall\, \mathbf{u}, \mathbf{v} \in \mathcal{F}.
\end{equation*}

Let $\mathbf{u} \in L^2(\Omega; \mathbb{R}^2)$ be a vector field satisfying the unit-norm constraint $|\mathbf{u}(x)| = 1$ for a.e. $x \in \Omega_2$. The mollified field, denoted by $e^{\tau \Delta} \mathbf{u}$, is obtained through the action of the heat semigroup. Since $G_\tau > 0 \, \, \, \text{on } \Omega_2$, the zero set $\mathcal{Z} = \left\{ x \in \Omega_2 \;\middle|\; \bigl(e^{\tau\Delta}\mathbf{u}\bigr)(x) = 0 \right\}$ is of Lebesgue measure zero. This ensures that the normalization mapping $\mathcal{P}(\mathbf{v}) := \mathbf{v}/|\mathbf{v}|$ is well-defined in the $L^2$ sense for the smoothed field. In our formulation, the regularization term $\frac{1}{\varepsilon^3} B(\phi)\mathbf{g}$ could, in principle, perturb this property. However, the consistency of the proposed approach is maintained in the asymptotic limit $\varepsilon \to 0$. Specifically, the term $\frac{1}{\varepsilon^3} B(\phi)$ acts as a localized weight that concentrates onto the interface $\partial \Omega_1$ in the sense of distributions. Consequently, any potential vanishing of $\mathbf{g}$ is restricted to a set of lower dimension, thereby preserving the a.e. well-posedness of the normalization step and ensuring the structural integrity of the numerical framework.
Based on this assumption, we establish the existence of a solution to the model associated with the modified energy functional.
\begin{theorem}[Existence of $\mathbf{u}$]
     For a given $\mathbf{g}, \phi$ and $\tau,\varepsilon>0$, the minimization problem of energy $E_{\tau,\epsilon}$ subject to a unit-norm constraint admits at least one solution.
\end{theorem}

\begin{proof}
    Denote
    $$
\tilde{\mathcal{C}} = \left\{ 
\tilde{\mathbf{u}} = (\tilde{u}_1, \tilde{u}_2) 
\;\middle|\; 
\mathbf{u} = (u_1, u_2) \in L^2(\Omega, \mathbb{R}^2),\,
\tilde{\mathbf{u}} = \frac{e^{\tau \Delta} \mathbf{u} + \frac{\tau}{\varepsilon^3} B(\phi) \mathbf{g}}{
\lvert e^{\tau \Delta} \mathbf{u} + \frac{\tau}{\varepsilon^3} B(\phi) \mathbf{g} \rvert},\,
\lvert \mathbf{u} \rvert = 1 \,  \text{ a.e.}
\right\}
$$
    and $cl(\tilde{\mathcal{C}})_{H^1(\Omega, \mathbb{R}^2)}$ as the closure of $\tilde{\mathcal{C}}$ in $H^1(\Omega, \mathbb{R}^2)$. Here we use the fact that $\tilde{\mathbf{u}}$ is actually a smooth function and its $H^1$ norm is bounded. Because $H^1(\Omega, \mathbb{R}^2)$ is compactly embedded in $L^2(\Omega, \mathbb{R}^2)$ and $cl(\tilde{\mathcal{C}})_{H_1(\Omega, \mathbb{R}^2)}$ is a closed and bounded subset in $H^1(\Omega, \mathbb{R}^2)$, $cl(\tilde{\mathcal{C}})_{H_1(\Omega, \mathbb{R}^2)}$ is then compact in $L^2(\Omega, \mathbb{R}^2)$. Moreover, because $E_{\tau,\epsilon}\left( \mathbf{\mathbf{u}} \right)$ is continuous in $L^2(\Omega, \mathbb{R}^2)$, there exists at least one $\mathbf{u}^\star \in cl(\tilde{\mathcal{C}})_{H_1(\Omega; \mathbb{R}^2)} $ such that
    \[     E_{\tau,\varepsilon}(\mathbf{u}^\star) = \inf_{\mathbf{u} \in cl(\tilde{\mathcal{C}})_{H^1(\Omega, \mathbb{R}^2)}} E_{\tau,\epsilon}(\mathbf{u}).     \] Further, it is straightforward to show $\left| \mathbf{u}^\star\right| = 1$ because $cl(\tilde{\mathcal{C}})_{H^1(\Omega, \mathbb{R}^2)} \subset cl(\mathcal{C})_{H^1(\Omega, \mathbb{R}^2)}$.

    For any $\mathbf{v} \in L^2(\Omega_2;\mathbb{R}^2)$ but not in $cl(\tilde{\mathcal{C}})_{H^1(\Omega; \mathbb{R}^2)}$, denote 
    \[
    \hat{\mathbf{v}} = \frac{e^{\tau \Delta} \mathbf{v} + \frac{\tau}{\varepsilon^3} B(\phi) \mathbf{g}}{
\lvert e^{\tau \Delta} \mathbf{v} + \frac{\tau}{\varepsilon^3} B(\phi) \mathbf{g} \rvert},
    \]
   
    we have
    \[
    -\left\langle \frac{1}{\tau}e^{\tau \Delta} \mathbf{v} + \frac{1}{\varepsilon^3} B(\phi) \mathbf{g}\,, \mathbf{v} \right\rangle \ge - \left\langle \frac{1}{\tau}e^{\tau \Delta} \mathbf{v} + \frac{1}{\varepsilon^3} B(\phi) \mathbf{g}\,, \hat{\mathbf{v}} \right\rangle
    \]
    
    and thus 

\begin{align*}
-\left\langle \frac{1}{2\tau}e^{\tau \Delta} \mathbf{v} + \frac{1}{\varepsilon^3} B(\phi) \mathbf{g}\,, \mathbf{v} \right\rangle 
&\ge -\left\langle \frac{1}{2\tau}e^{\tau \Delta} \mathbf{v} + \frac{1}{\varepsilon^3} B(\phi) \mathbf{g}\,, \mathbf{v} \right\rangle - \left\langle \frac{1}{\tau}e^{\tau \Delta} \mathbf{v} + \frac{1}{\varepsilon^3} B(\phi) \mathbf{g}\,, \hat{\mathbf{v}} - \mathbf{v} \right\rangle \\
&\ge -\left\langle \frac{1}{2\tau}e^{\tau \Delta} \hat{\mathbf{v}} + \frac{1}{\varepsilon^3} B(\phi) \mathbf{g}\,, \hat{\mathbf{v}} \right\rangle.
\end{align*}

The final inequality follows from the property that the linearization of a concave functional always bounds the functional from above. Consequently, we obtain $E_{\tau ,\varepsilon}(\mathbf{v}) \ge E_{\tau ,\varepsilon}(\hat{\mathbf{v}})$, which ensures that the minimum is indeed achieved within $cl(\tilde{\mathcal{C}})_{H^1(\Omega; \mathbb{R}^2)}$.
    
\end{proof}

\section{Derivation of the numerical scheme}
\label{sec:Derivation}
In this section, we address the minimization problem of energy $E_{\tau,\varepsilon}$ using linear sequential iterations and establish the monotonic energy decay property associated with this iteration scheme. To minimize a concave functional \( E_{\tau,\epsilon}(\mathbf{v}) \) over the constrained set, we propose an iterative method based on the supporting property of concave functionals. At each iteration \( k \), the functional  is approximated by its first-order linear expansion (supporting hyperplane) at the current iterate \( \mathbf{u}_k \):
\begin{equation*}
    l_{\mathbf{u}_k}(\mathbf{v}) = E_{\tau,\epsilon}(\mathbf{u}_k) 
    + \left\langle \frac{\delta E_{\tau,\epsilon}}{\delta \mathbf{u}}(\mathbf{u}_k), \mathbf{v} - \mathbf{u}_k \right\rangle,
\end{equation*}
Subsequently, the computation of \( \mathbf{u}_{k+1} \) is reformulated as a   minimal problem for the energy functional \( l_k(\mathbf{u}) \). Since \( \mathbf{u}_{k+1} \) must also lie within the constraint set, we obtain the following formulation:
\begin{equation*}
\mathbf{u}_{k+1} 
= \arg\min_{\mathbf{u} \in \mathcal{C} } l_{\mathbf{u}_k}(\mathbf{u}) 
= \arg\min_{\mathbf{u} \in \mathcal{C}} 
- \int_{\Omega_2} \left\langle \mathbf{u},\frac{1}{\tau} e^{\tau \Delta} \mathbf{u}_k 
+ \frac{1}{\varepsilon^3} B(\phi) \mathbf{g} \right\rangle \,\mathrm{d}x.
\end{equation*}
Since the \( \ell_2 \) norm of the vector function \( \mathbf{u} = (u_1, u_2) \) is always equal to 1, and by the properties of the inner product, the minimizer of the functional \( l_k \) with the unit norm constraint is given by the projection of \( e^{\tau \Delta} \mathbf{u}_k + \frac{2}{\epsilon^3} B(\phi) \mathbf{g} \) onto \( \mathcal{C} \), which was given by \eqref{eq:iteration}. 
\begin{equation}
\label{eq:iteration}
\mathbf{u}_{k+1} = \frac{
e^{\tau \Delta} \mathbf{u}_k + \dfrac{\tau}{\varepsilon^3} B(\phi) \mathbf{g}
}{
\bigl| e^{\tau \Delta} \mathbf{u}_k + \dfrac{\tau}{\varepsilon^3} B(\phi) \mathbf{g} \bigr|
}
\end{equation}
By iterating the following process to steady state, the final vector field can be obtained. The complete iterative process is described in \cref{alg:iterative_projection}.
\begin{algorithm}
\caption{Diffusion and normalization algorithm}
\label{alg:iterative_projection}
\begin{algorithmic}[1]
\REQUIRE Initial field $\mathbf{u}_0$, function $g$, function $\phi$, parameters $\tau$, $\epsilon$, tolerance $\text{tol}$
\ENSURE Final solution $\mathbf{u}$

\STATE Set $k := 0$
\STATE Set $\delta := +\infty$
\STATE Set $\mathbf{u} := \mathbf{u}_0$
\WHILE{$\delta > \text{tol}$}
    \STATE Compute $\mathbf{v} := e^{\tau \Delta} \mathbf{u}_k + \dfrac{2}{\epsilon^3} B(\phi)\, \mathbf{g}$
    \STATE Set $\mathbf{w} := \dfrac{\mathbf{v}}{|\mathbf{v}|}$
    \STATE Compute $\delta := \|\mathbf{u} - \mathbf{w}\|$
    \STATE Set $k := k + 1$
    \STATE Set $\mathbf{u} := \mathbf{w}$
\ENDWHILE
\RETURN {$\mathbf{u}$}
\end{algorithmic}
\end{algorithm}
Subsequently, we further ensure that, for any combination of parameters $\varepsilon$ and $\tau$, the iterative sequence satisfies the energy-decreasing property and converges in a finite number of steps.
\begin{theorem}[Energy Diminishing Property]
\label{thm:energy_decay}
Let $\Omega_2, \mathbf{g}$, and $\phi$ be defined as in lemma \ref{lemma:energy}. 
For any $\mathbf{u} \in L^2(\Omega_2; \mathbb{R}^2)$ satisfying $|\mathbf{u}| = 1$ a.e., let $\mathbf{v}$ be the update defined by the iterative scheme \eqref{eq:iteration}:
\[
\mathbf{v} = \frac{e^{\tau \Delta} \mathbf{u} + \frac{\tau}{\varepsilon^3} B(\phi) \mathbf{g}}{\lvert e^{\tau \Delta} \mathbf{u} + \frac{\tau}{\varepsilon^3} B(\phi) \mathbf{g} \rvert}.
\]
Assuming the denominator is non-zero a.e., the energy functional $E_{\tau,\epsilon}$ satisfies the monotonicity property
\[
E_{\tau,\epsilon}(\mathbf{v}) \leq E_{\tau,\epsilon}(\mathbf{u}).
\]
\end{theorem}

\begin{proof}
    Since $\mathbf{u} \in \mathcal{C}$ and the linear functional $l_{\mathbf{u}}$ is defined on the Hilbert space $L^2(\Omega_2;\mathbb{R}^2)$, it follows that for any $\mathbf{v} \in L^2(\Omega_2;\mathbb{R}^2)$, the inequality
\begin{equation*}
l_{\mathbf{u}}(\mathbf{v}) \geq E_{\tau,\epsilon}(\mathbf{v})
\end{equation*}
holds. Consequently, we obtain the following chain of inequalities:
\begin{equation*}
E_{\tau,\epsilon}(\mathbf{v}) \leq l_{\mathbf{u}}(\mathbf{v}) \leq l_{\mathbf{u}}(\mathbf{u}) = E_{\tau,\epsilon}(\mathbf{u}),
\end{equation*}
which establishes an upper bound of $E_{\tau,\epsilon}(\mathbf{v})$ in terms of the functional $l_{\mathbf{u}}$ evaluated at $\mathbf{u}$.
\end{proof}
The equality holds only when $\mathbf{v} = \mathbf{u}$ . Since the energy functional has a low bound, the algorithm converges to a stationary solution in finite steps. 

\section{Implementation}
\label{sec:Implementation}

This section describes the practical implementation of the proposed framework. The emphasis is on the efficient computation of cross fields via the exponential of the Laplacian formulation, which constitutes the core part of the algorithmic implementation. To ensure reproducibility and facilitate practical adoption, we also summarize the subsequent stages of the pipeline, including separatrix extraction guided by the computed cross field, domain decomposition, and the construction of structured quadrilateral meshes. The presentation is intended to connect the theoretical development with its numerical realization in a clear and systematic manner.

\subsection{Computation of Cross-Fields}

We describe the numerical realization of the three ingredients in \cref{alg:iterative_projection}: (i) the phase-field $\phi$, (ii) an extension $\mathbf{g}$ of the boundary data, and (iii) the diffusion operator $e^{\tau\Delta}$ acting on vector fields. The functions $\phi$ and $\mathbf{g}$ encode the geometry and boundary alignment information, while $e^{\tau\Delta}$ provides the linear smoothing step in each iteration. Throughout, the computational domain $\Omega_2$ is taken to be a square equipped with periodic boundary conditions so that convolutions and heat diffusion can be implemented spectrally.

\paragraph{Phase-field construction}
As in Section~\ref{sec:CDDM}, we define $\phi$ as a mollification of the indicator function $\chi_{\Omega_1}$ by a Gaussian kernel,
\[
\phi = G_\varepsilon * \chi_{\Omega_1},
\qquad
G_\varepsilon(x) = \frac{1}{(2\pi\varepsilon^2)^{d/2}}
\exp\!\left(-\frac{|x|^2}{2\varepsilon^2}\right),
\]
where $\varepsilon>0$ controls the width of the diffuse interface. This convolution-based construction yields a smooth approximation of $\chi_{\Omega_1}$ on $\Omega_2$ and is naturally compatible with FFT-based evaluation.

\paragraph{FFT evaluation of the diffusion step}
Let $u:\Omega_2\to\mathbb{R}^2$ be a vector field. On the periodic square $\Omega_2$, the heat semigroup $e^{\tau\Delta}$ is diagonalized in the Fourier basis. Denoting by $\xi$ the Fourier frequency, the diffusion step admits the representation
\[
e^{\tau\Delta}u
=
\mathrm{IFFT}\!\left(
e^{-\tau|\xi|^2}\,\mathrm{FFT}(u)
\right),
\]
applied component-wise to $u$. Although the theoretical model in~\eqref{eq:inte} is posed with boundary conditions associated with $\Omega_1$, the spectral realization on $\Omega_2$ introduces only a weak influence from $\partial\Omega_2$ provided that $\Omega_2$ is chosen sufficiently large. In the subsequent experiments, any discrepancy remains localized near $\partial\Omega_2$ and does not affect the computed field within $\Omega_1$.

\paragraph{Boundary-data extension}
A remaining practical issue is the construction of an extension field $\mathbf{g}$ entering the diffuse-domain penalty term. Let $\Gamma=\partial\Omega_1$ and define a tubular neighborhood
\[
\Omega_\delta
=
\{x\in\mathbb{R}^2:\operatorname{dist}(x,\Gamma)<\delta\}.
\]
A standard diffuse-domain prescription extends boundary data along normal directions, which may be expressed in terms of the closest-point projection
\[
\mathrm{cp}_\Gamma(x)
=
\arg\min_{y\in\Gamma}\|x-y\|_2,
\qquad
g^e(x)=g\!\left(\mathrm{cp}_\Gamma(x)\right),
\quad x\in\Omega_\delta.
\]
While this construction is natural when $\Gamma$ is smooth, it can lose regularity near $C^0$ corners, where $\mathrm{cp}_\Gamma$ may be set-valued and discontinuities can be induced in $g^e$. Since cross-field computation requires a sufficiently regular driving field, we replace $g^e$ by a smoothed approximation obtained via mollification.


To compute $g^e(x)$, we approximate $\Gamma$ by a polygonal curve $\Gamma_h$ obtained from an ordered set of sample points $\{\mathbf{p}_i\}_{i=1}^N$. For each $x\in\Omega_\delta$, let $\mathrm{cp}_{\Gamma_h}(x)$ denote the closest-point projection of $x$ onto $\Gamma_h$ (computed by closest-segment projection). We define an initial extension by evaluating the boundary data at the projected point, with linear interpolation along the containing segment:
\[
g^e_h(x) := g\bigl(\mathrm{cp}_{\Gamma_h}(x)\bigr).
\]
In general, $g^e_h$ may have limited regularity, particularly near $C^0$ corners where the closest-point map can be non-smooth. To obtain a sufficiently regular driving field for the diffusion--normalization iteration, we apply a Gaussian mollification and set
\[
 \tilde g^e_h(x)
:=
(G_\sigma * g^e_h)(x)
=
\int_{\Omega_\delta} G_\sigma(x-y)\,g^e_h(y)\,\mathrm{d}y,
\]
where $\sigma>0$ is a smoothing scale chosen commensurate with the diffuse-interface width. In implementation, $g^e_h$ is extended by zero outside $\Omega_\delta$ to the periodic box $\Omega_2$, and the convolution is evaluated by FFT. This construction yields an extension field with improved interior regularity while preserving the boundary-alignment information required by the diffuse-domain penalty term. The complete procedure for boundary-data extension is summarized in~\cref{alg:vectorfield}.

\begin{algorithm}
\caption{Vector field extension via closest segment interpolation}
\label{alg:vectorfield}
\begin{algorithmic}[1]
\REQUIRE Discrete closed curve $\mathcal{C} = \{P_1, \dotsc, P_N\} \subset \mathbb{R}^2$
\REQUIRE Associated vectors $\{\vec{v}_1, \dotsc, \vec{v}_N\} \subset \mathbb{R}^2$
\REQUIRE Tubular neighborhood $\Omega$ around the curve
\ENSURE A vector field $\tilde g^e_h(p)$ defined for all $p \in \Omega$

\STATE Append $P_{N+1} := P_1$, $\vec{v}_{N+1} := \vec{v}_1$ to close the curve

\FOR{each point $p \in \Omega$}
  \STATE Find closest projection $q$ of $p$ onto the polyline $\mathcal{C}$ via Closest Segment Projection
  \STATE Find interpolation weight $\alpha \in [0,1]$ such that $q = (1 - \alpha) P_i + \alpha P_{i+1}$
  \STATE Set $g^e_h(x) \gets (1 - \alpha)\vec{v}_i + \alpha \vec{v}_{i+1}$
\ENDFOR

\STATE \quad Let $ \tilde g^e_h(p)
:=
(G_\sigma * g^e_h)(p)$

\RETURN $\tilde g^e_h(p)$
\end{algorithmic}
\end{algorithm}

The above constructions provide all inputs for the iterative diffusion--normalization scheme, and yield an implementation that is fully compatible with structured grids and FFT-based acceleration.

\subsection{Tracing Streamlines}\label{sec:Tracing}

To construct a quadrilateral layout from a computed cross field, one must (i) localize the field singularities and (ii) initialize the separatrices issuing from these defects. In the Ginzburg--Landau regime, the Dirichlet energy density is strongly concentrated near singularities, which motivates an energy-based detection criterion. We therefore evaluate a discrete Dirichlet energy $E(i,j)$ on each cell (via linear interpolation of the field over the cell) and designate $(i,j)$ as a candidate singular cell if $E(i,j)$ attains a strict local maximum on the associated $3\times3$ neighborhood,
\[
E(i,j) > E(i+m,j+n)
\qquad
\forall (m,n)\in\{-1,0,1\}^2\setminus\{(0,0)\}.
\]

Given a candidate cell, we next localize the singularity as the zero of the interpolated representation field. Let $\mathbf{s}_0$ denote the singular point within the cell, obtained by solving for the vanishing of the bilinearly interpolated complex-valued representation. To initialize separatrix tracing, we examine the field along the cell boundary. Denote the cell vertices by $\{\mathbf{s}_1,\mathbf{s}_2,\mathbf{s}_3,\mathbf{s}_4\}$ and consider an edge $E=[\mathbf{s}_1,\mathbf{s}_2]$. For a point $P\in E$, evaluation of the cross field yields four symmetric direction vectors $\{\mathbf{v}_k(P)\}_{k=1}^4$. An edge point $P$ is identified as an intersection with a separatrix if one branch aligns with the radial direction from the singularity, namely,
\[
\exists\, k\in\{1,2,3,4\}
\quad \text{such that} \quad
\frac{\mathbf{v}_k(P)}{\|\mathbf{v}_k(P)\|}
=
\frac{P-\mathbf{s}_0}{\|P-\mathbf{s}_0\|}.
\]
The same test is applied to the remaining edges of the cell. The resulting intersection points provide the initial conditions for separatrix integration, specifying both the starting locations and the associated tangent directions. As shown in~\cref{fig:initial_condition}, the detected configurations correspond to singularities of index $+\tfrac{1}{4}$ and $-\tfrac{1}{4}$, which in turn induce irregular mesh vertices of valence $3$ and $5$, consistent with the topological characterization in Section~\ref{sec:crossfield}.

\begin{figure}[t!]
    \centering

    \begin{subfigure}[t]{0.48\textwidth}
        \centering
        \includegraphics[width=\linewidth]{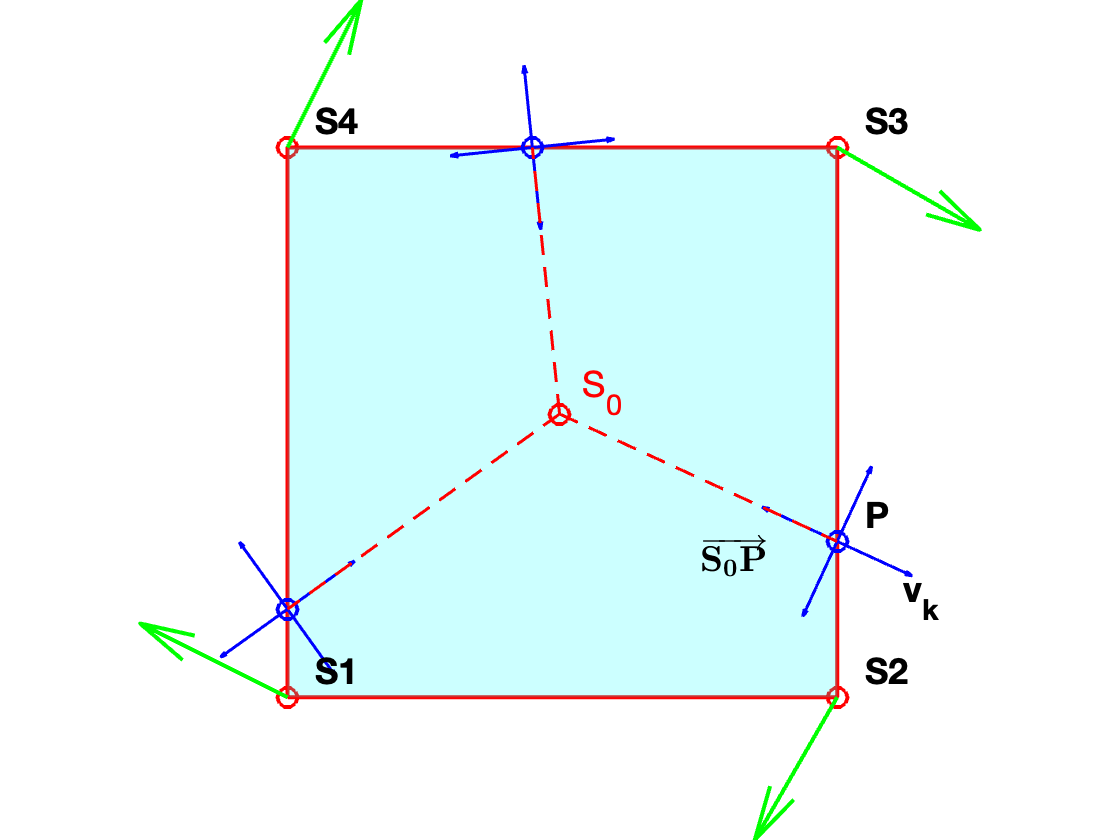}
        \label{fig:stream_1}
    \end{subfigure}
    \begin{subfigure}[t]{0.42\textwidth}
        \centering
        \includegraphics[width=\linewidth]{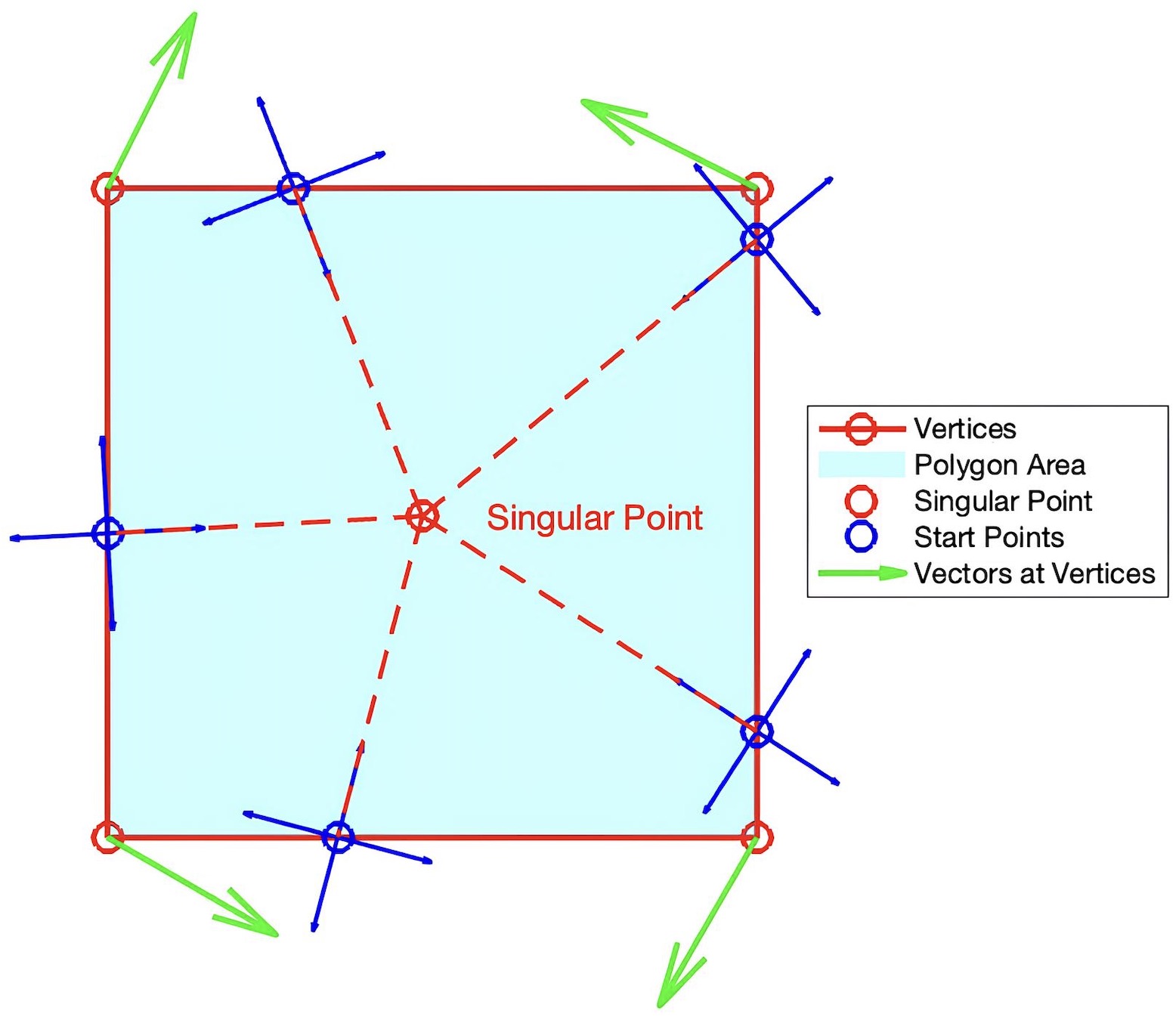}
        \label{fig:stream_2}
    \end{subfigure}

    \caption{The process of finding initial conditions for streamline propagation in triangular and quadrilateral cells. See Section~\ref{sec:Tracing}.}
    \label{fig:initial_condition}
\end{figure}

\begin{figure}[t!]
  \centering
  \label{fig:propogation}\includegraphics[scale=0.15]{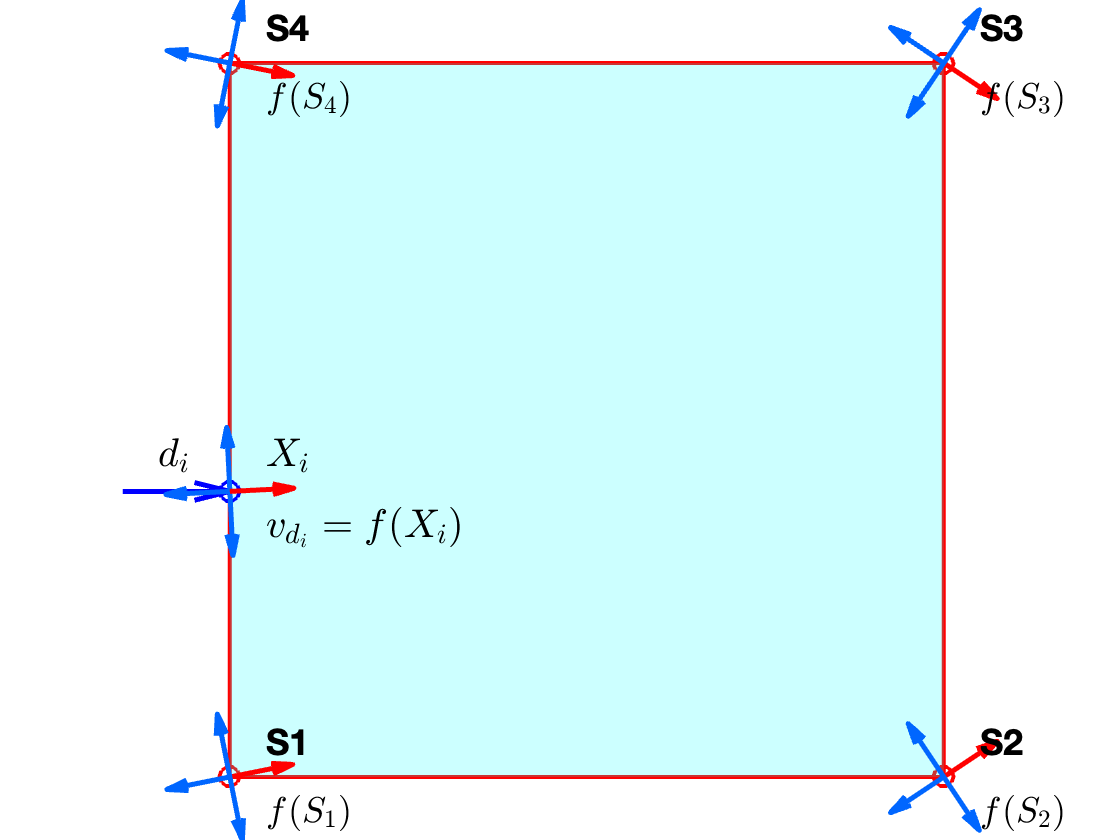}
  \caption{Streamline propagation within a structured quadrilateral cell. See Section~\ref{sec:Tracing}.}
  \label{fig:Propagation}
\end{figure}

Unlike a single-valued vector field, a cross field is intrinsically multi-valued: at each point it specifies four directions related by $\pi/2$ rotations. To obtain a well-defined streamline, branch consistency is enforced through the selection rule~\eqref{eq:streamline condition}, which promotes directional continuity along the trajectory. As illustrated in~\Cref{fig:Propagation}, let $X_i$ denote the current position and let $d_i$ be the propagation direction at step $i$. The cross field value $f(X_i)$ is evaluated by interpolation of the vertex data, yielding four candidate directions. The direction for the next step is chosen as the branch that minimizes the angular deviation from $d_i$, and the trajectory is advanced to $X_{i+1}$ using a standard time-stepping scheme (e.g., explicit Euler or Runge--Kutta). Integration is terminated when the streamline exits the domain or enters a prescribed neighborhood of a singularity. For alternative streamline-tracing strategies with enhanced robustness, see~\cite{campen2015quantized,RobustStreamline1}.

\subsection{Block mesh generation}

After the separatrices partition the domain into quadrilateral-like blocks, a structured quadrilateral mesh is generated on each block by a harmonic mapping approach~\cite{li2001moving,Remacle2010}. The basic idea is to construct a smooth bijective map from a logical square to the physical block by solving Laplace equations; the resulting coordinate functions are discrete harmonic and typically distribute grid lines smoothly, thereby reducing element distortion.

Let $D=[0,1]^2$ be the logical domain with coordinates $\mathbf{u}=(u,v)$, and let $\Omega$ denote a physical block. The mapping $x:D\to\Omega$ is written as $x(\mathbf{u})=(x_1(u,v),x_2(u,v))$, where $x_1$ and $x_2$ solve the decoupled boundary value problems
\[
\Delta x_1(\mathbf{u})=0,\qquad \Delta x_2(\mathbf{u})=0,\qquad \mathbf{u}\in D,
\]
with Dirichlet data prescribed by the geometry of $\partial\Omega$. Denote the four boundary curves of $\Omega$ by $\Gamma_k$ and the corresponding edges of $\partial D$ by $\gamma_k$ for $k=1,\dots,4$. A boundary correspondence is imposed by
\[
x\big|_{\gamma_k}=\mathbf{g}_k,\qquad \mathbf{g}_k:[0,1]\to \Gamma_k,
\]
where $\mathbf{g}_k$ is chosen to parameterize $\Gamma_k$ (e.g., by arc length) and to match consistently across adjacent blocks. Once $x$ is computed, the image of a uniform tensor-product grid on $D$ provides a structured quadrilateral mesh on $\Omega$.

In implementation, each edge is sampled by $N$ points to define the discrete boundary values. The Laplace equations are then discretized on $D$ using standard second-order finite differences, yielding sparse linear systems for $\{x_{ij}\}$ and $\{y_{ij}\}$. The resulting nodal coordinates $\{(x_{ij},y_{ij})\}$ constitute the vertices of the quadrilateral mesh on the physical block.

\section{Numerical experiments}
\label{sec:Numerical}

This section reports numerical experiments designed to assess the proposed method with respect to (i) energy stability of the iteration, (ii) sensitivity to the parameters $\varepsilon$ and $\tau$ and the resulting singularity configurations, (iii) robustness under geometric irregularities (high curvature and sharp corners), and (iv) performance on complex domains relevant to practical quadrilateral meshing. We first verify the monotone decay of the discrete energy predicted by the analysis. We then quantify how parameter scaling affects singularity locations and the induced block decomposition, and compare the limiting behavior with established sharp-interface solvers. Finally, we demonstrate robustness and mesh quality on domains with nonsmooth boundaries and on representative nontrivial geometries.

\subsection{Unconditional energy stability of the iterative scheme}
\label{sec:UnconditionalDecrease}

As shown in Section~\ref{sec:Derivation}, \cref{alg:iterative_projection} is unconditionally energy stable for any $\epsilon>0$ and $\tau>0$. Here, $\tau$ is the diffusion time step in the exponential of the Laplacian $e^{\tau\Delta}$, controlling the smoothing strength, while $\epsilon$ specifies the diffuse-interface width and scales the boundary-penalty term in the convolution-based diffuse domain method.

To corroborate the theory, we consider the unit disk
\[
\Omega_1=\{x\in\mathbb{R}^2:\ |x|\le 1\},
\]
embedded in the periodic computational domain $\Omega_2=[-\pi,\pi]^2$. The monotone decay result applies to the modified energy $E_{\tau,\epsilon}$ posed on $\Omega_2$; in addition, we report the Dirichlet energy restricted to $\Omega_1$ as a diagnostic for the physical field quality in the target domain. The iteration is carried out on $\Omega_2$ with parameters scaled by the grid size $h$ via
\[
\epsilon=\tau=\alpha h,
\]
and we test $\alpha\in\{1,\,1/4,\,1/16,\,1/32,\,1/64\}$.

\begin{figure}[t!]
  \centering
  \begin{subfigure}[b]{0.46\textwidth}
    \label{fig:logDeng}
    \centering
    \includegraphics[width=\linewidth]{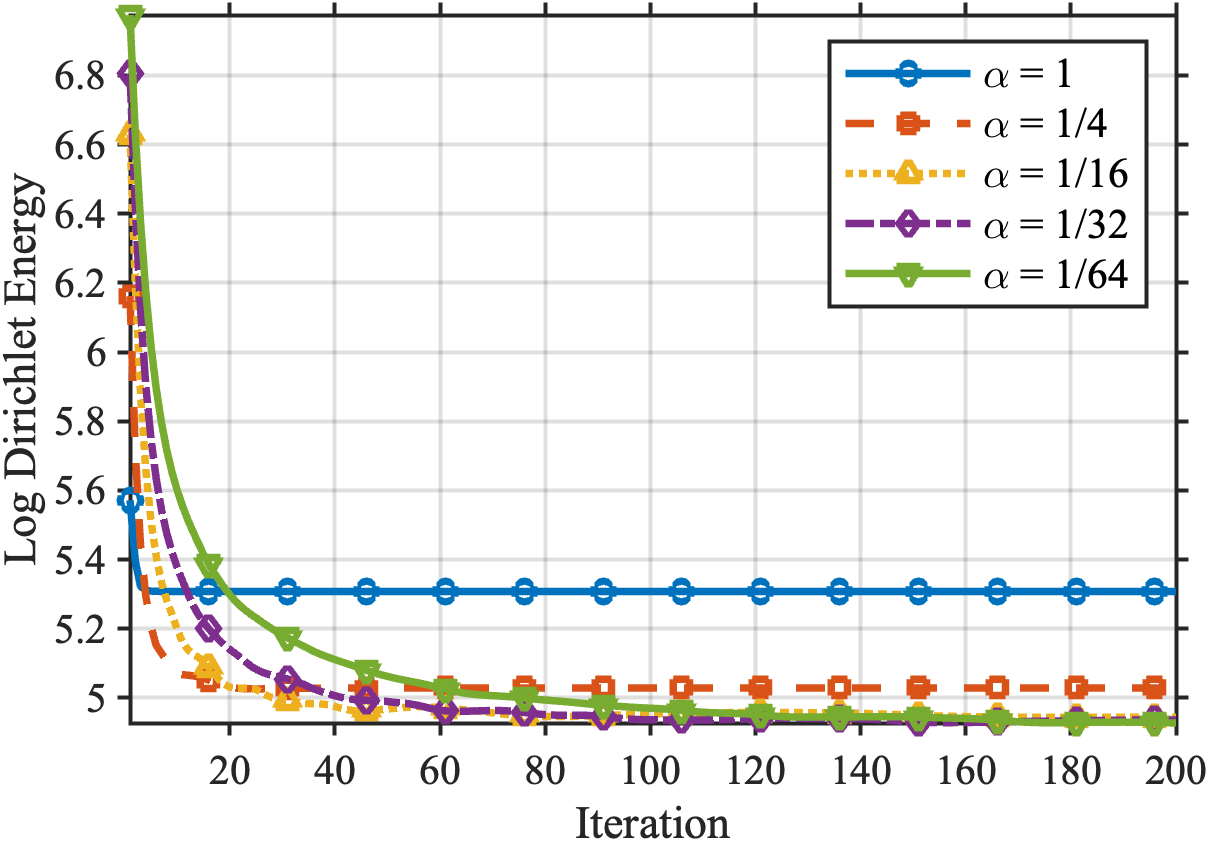} 
    \caption*{(a)} 
  \end{subfigure}
  \hfill
  \begin{subfigure}[b]{0.45\textwidth}
    \label{fig:logAeng}
    \centering
    \includegraphics[width=\linewidth]{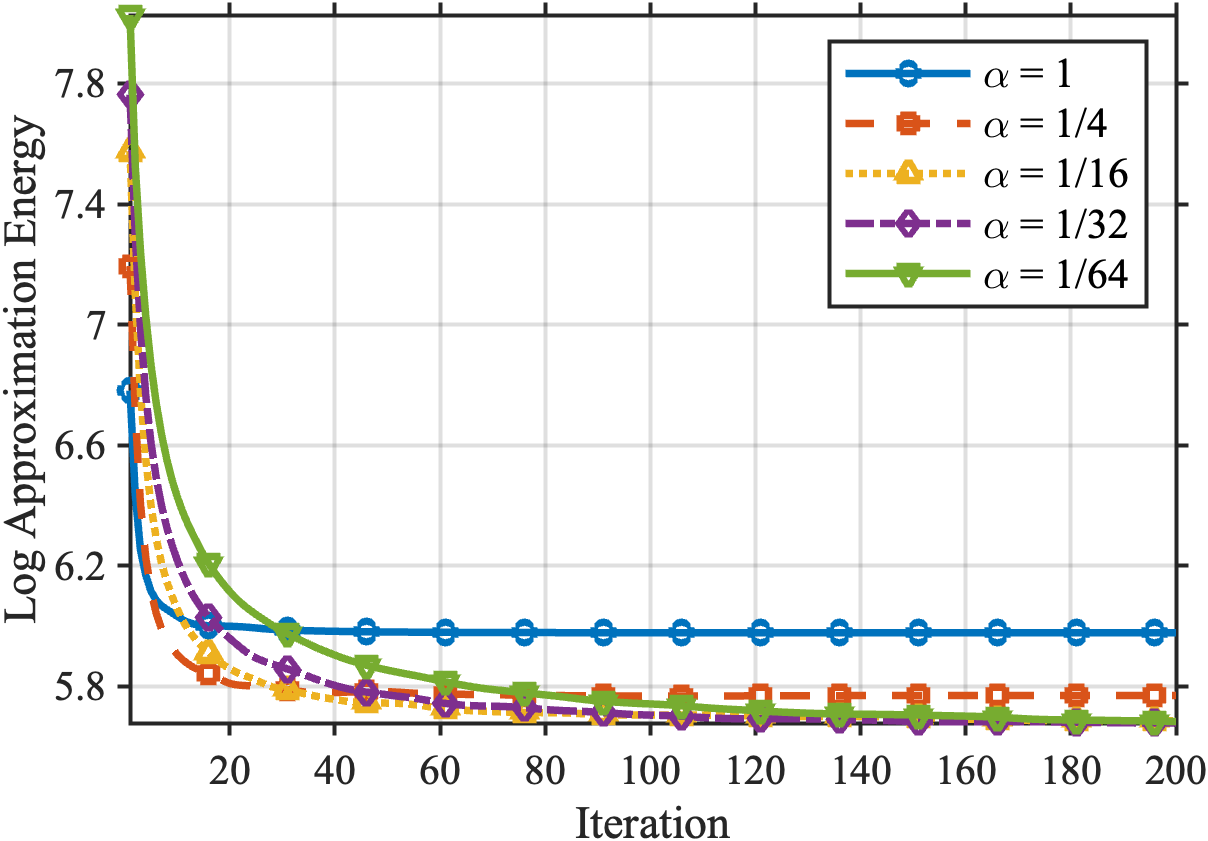} 
    \caption*{(b)}
  \end{subfigure}
  \caption{
(a) Dirichlet energy distribution within the subdomain $\Omega_1$ for various values of $\alpha = 1$, $1/4$, $1/16$, $1/32$, and $1/64$, under the condition $\varepsilon = \tau = \alpha h$.  
(b) Decay of the approximation energy over the background domain $\Omega_2$ for the same set of parameters. See Section~\ref{sec:UnconditionalDecrease}.
} 
  \label{fig:comparison}
\end{figure}
As shown in~\cref{fig:comparison}, the modified energy $E_{\epsilon,\tau}$ decays monotonically for all tested parameter choices, in agreement with the theoretical guarantee. In addition, the Dirichlet energy restricted to $\Omega_1$ decreases steadily throughout the iteration, although no stability result is claimed for this quantity. Taken together, these observations indicate that the method is not only globally energy diminishing on $\Omega_2$, but also yields a progressively smoother and more consistent field within the physical domain $\Omega_1$ in the tested regime.

\subsection{Influence of parameters and energy on singularities and decomposition} \label{sec:influence}

Section~\ref{sec:UnconditionalDecrease} establishes that the iteration is energy diminishing for any $\epsilon>0$ and $\tau>0$. This property, however, does not imply convergence to a unique steady state. In practice, different parameter choices can lead to different limiting fields and different final energy levels. This behavior is consistent with a nonconvex energy landscape, in which multiple locally stable configurations exist; each configuration corresponds to a distinct arrangement of singularities and induces a different separatrix graph and block decomposition.

In this subsection, the sensitivity of the computed singular structure and the resulting decomposition to $(\epsilon,\tau)$ is examined numerically. We also consider the regime $\epsilon,\tau\to 0$ and show that the solution approaches the solution of the original formulation~\eqref{eq:originproblem}, providing evidence for asymptotic consistency.

\begin{figure}[t!]
  \centering

  \subfloat[$\varepsilon =\tau= 0.5h$]{%
    \includegraphics[width=0.2\linewidth]{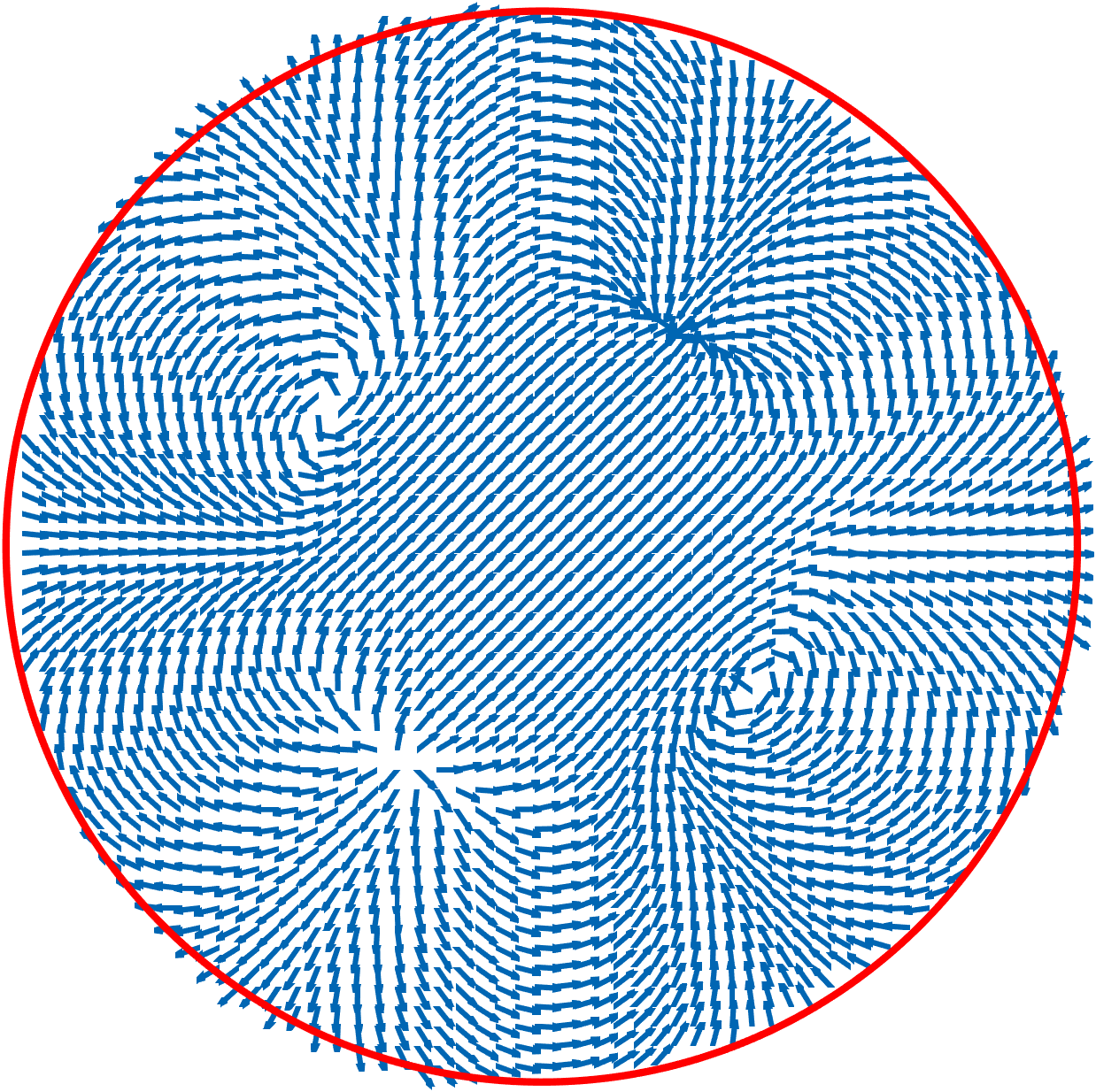}
  }
  \hspace{0.02\linewidth}
  \subfloat[$\varepsilon =\tau = 0.25h$]{%
    \includegraphics[width=0.2\linewidth]{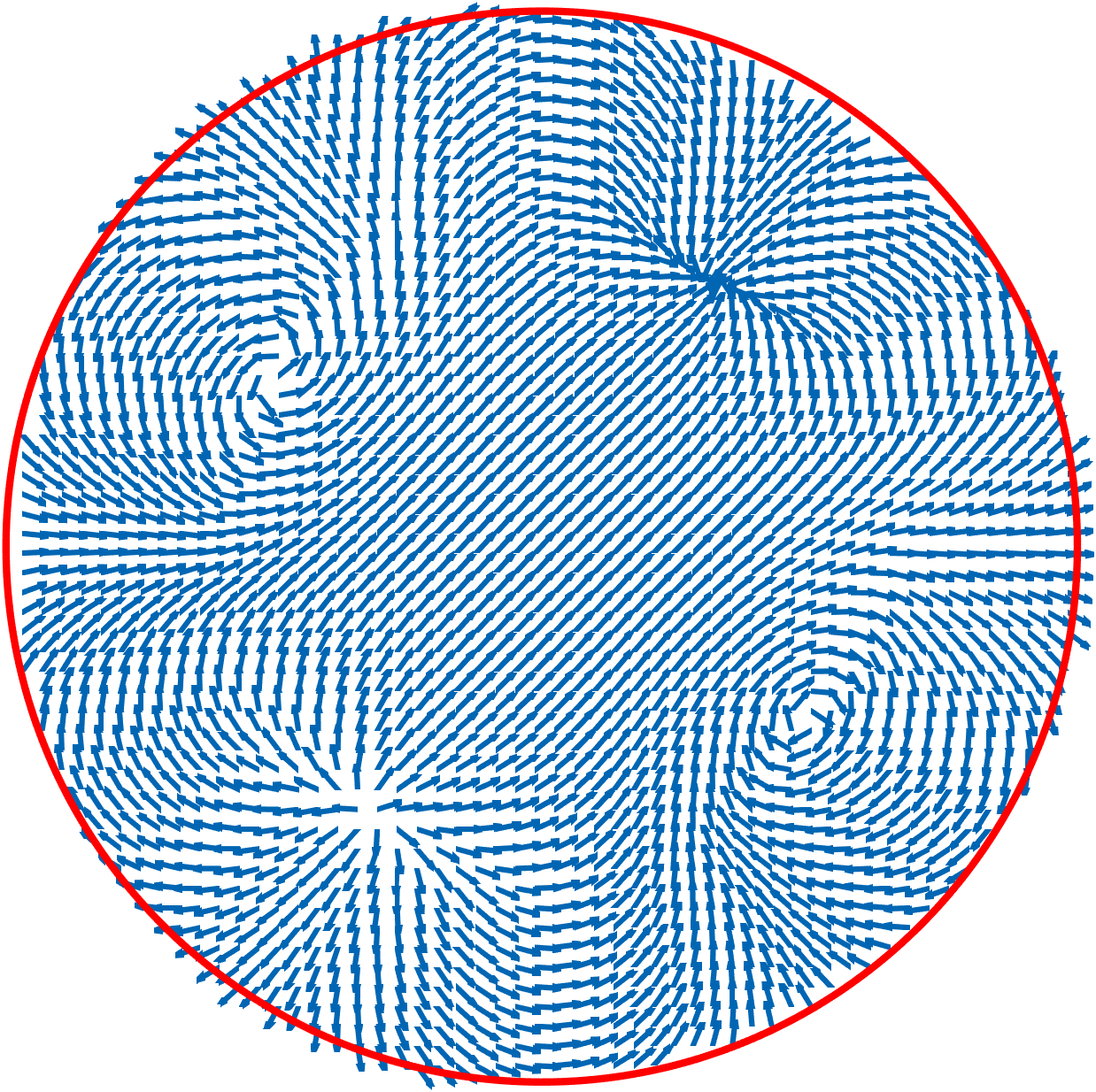}
  }
  \hspace{0.02\linewidth}
  \subfloat[$\varepsilon = \tau = 0.125h$]{%
    \includegraphics[width=0.2\linewidth]{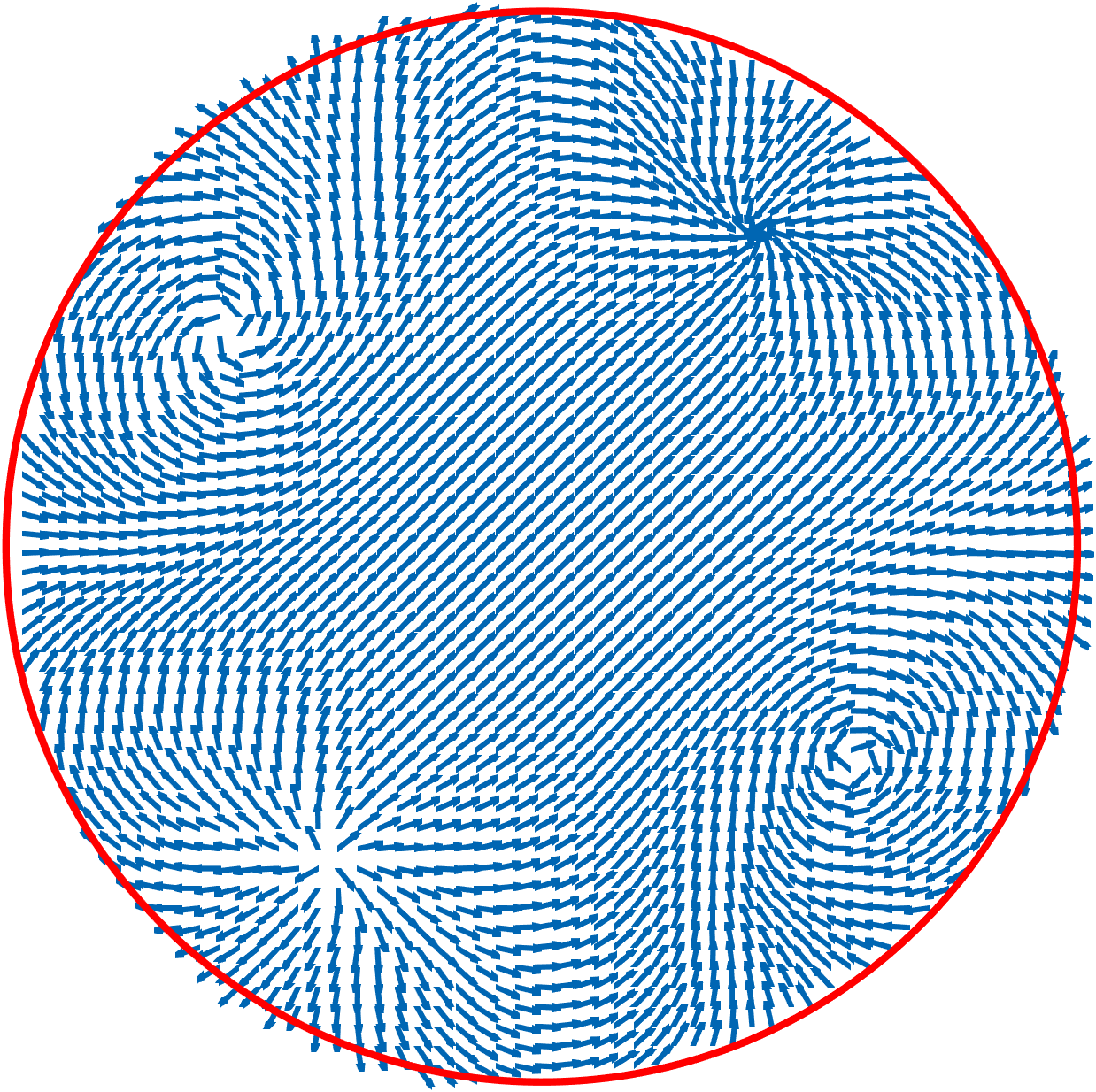}
  }
  \hspace{0.02\linewidth}
  \subfloat[$\varepsilon = \tau=0.0625h$]{%
    \includegraphics[width=0.2\linewidth]{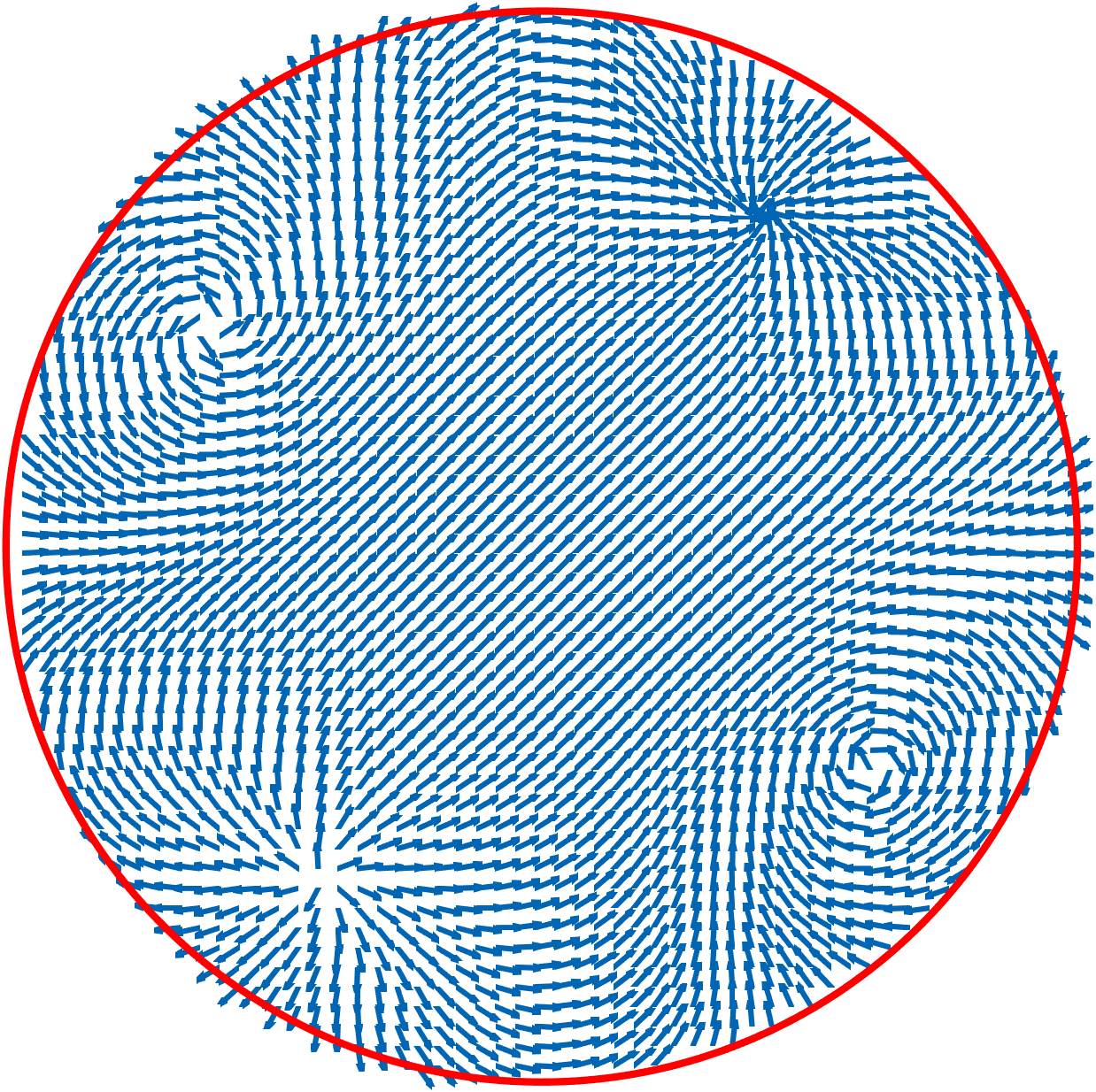}
  }

  \caption{
    Evolution of singularity regions in the computed vector field under different values of $\alpha$, with parameters set as $\varepsilon = \tau = \alpha h$. The four subfigures correspond to $\alpha = 0.5$, $0.25$, $0.125$, and $0.0625$, respectively. As $\alpha$ decreases, the singularities become more diffuse, and the radius of the singular regions increases. This illustrates the influence of the Ginzburg--Landau parameters on the structure of the vector field. See Section~\ref{sec:influence}.
  }
  \label{fig:Vector_field_singularties}
  
\end{figure}

In the second experiment, the geometry is unchanged from Section~\ref{sec:UnconditionalDecrease}, with the same physical domain $\Omega_1$ and computational domain $\Omega_2$. We consider the coupled scaling
\[
\epsilon=\tau=\alpha h,
\]
and examine the effect of parameter refinement using $\alpha\in\left\{\tfrac12,\tfrac14,\tfrac18,\tfrac1{16}\right\}$. The resulting directional fields are shown in~\Cref{fig:Vector_field_singularties}, illustrating the dependence of the steady-state singularity pattern on $\alpha$.

Across all tested values, the steady states retain four singularities, but their locations vary systematically with $\alpha$. As $\alpha$ decreases, the singularities migrate radially outward toward $\partial\Omega_1$, indicating that the coupled scale $\epsilon=\tau=\alpha h$ influences not only the effective diffusion length but also the equilibrium placement of topological defects.

To assess the sharp-interface limit, a reference solution is computed on $\Omega_1$ using the Merriman--Bence--Osher (MBO) scheme \cite{Osting2019}, which provides an operator-splitting approximation of the variational model~\eqref{eq:originproblem}. Convergence of MBO-type diffusion generated method to the associated gradient flow has been established in~\cite{laux2019analysis}. In our implementation, the MBO scheme is discretized by $P_1$ finite elements in space and explicit Euler time stepping.

Comparing~\Cref{fig:Vector_field_singularties} with the MBO solution in~\Cref{fig:MBOSolution} shows that, in the sharp-interface regime, the singularities lie at approximately $r\approx 0.85$. This is consistent with the analysis in~\cite{BEAUFORT2017219}, where local minimizers of the renormalized energy on circular domains were characterized in terms of the singularity radius. Moreover, as $\alpha$ decreases (and hence $\epsilon,\tau\to 0$), the singularity radii produced by the proposed model approach the MBO-predicted value, supporting the asymptotic consistency of the proposed regularization with~\eqref{eq:originproblem}.

\begin{figure}[t!]
  \centering
  \subfloat[]{%
    \includegraphics[width=0.3\linewidth]{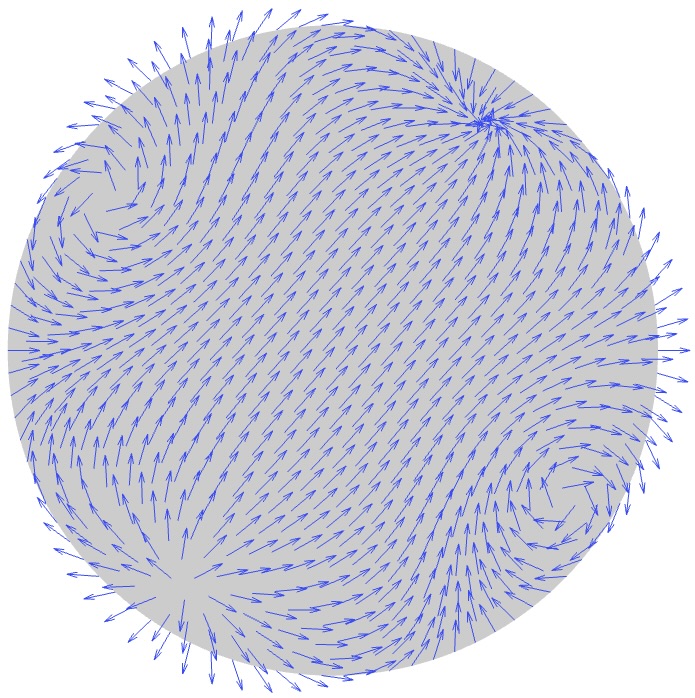}
    \label{fig:MBO_vector}
  }
  \hspace{0.02\linewidth} 
  \subfloat[]{%
    \includegraphics[width=0.35\linewidth]{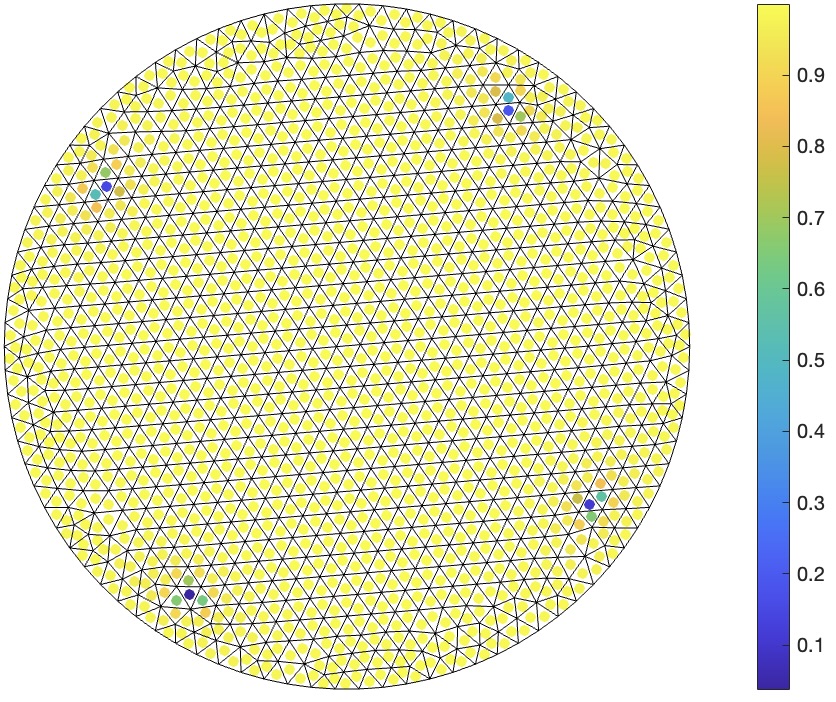}
    \label{fig:fig:MBO_norm_in_tri}
  }
  \caption{(a) Direction field computed by the MBO method. (b) Magnitude of the vectors evaluated at the triangle barycenters of the mesh used in the MBO computation, indicating the locations of singularities. See Section~\ref{sec:influence}.}
  \label{fig:MBOSolution}
\end{figure}

\begin{figure}[t!]
  \centering
  \subfloat[]{%
    \includegraphics[width=0.45\linewidth]{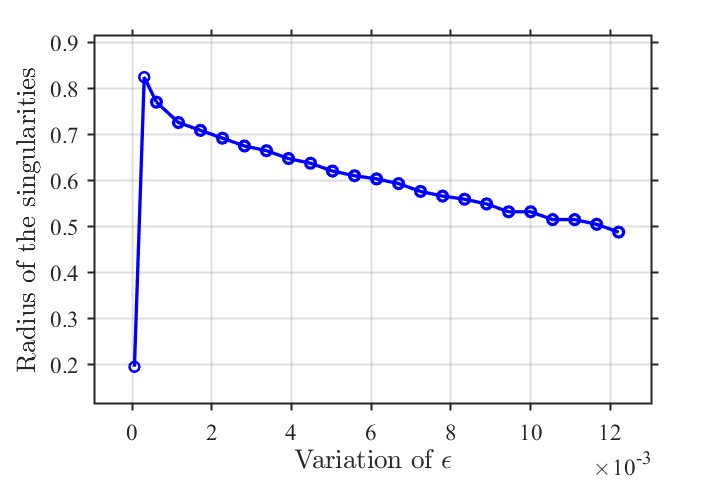}
    \label{fig:Radius_singularity_siam}
  }
  \hfill
  \subfloat[]{%
    \includegraphics[width=0.45\linewidth]{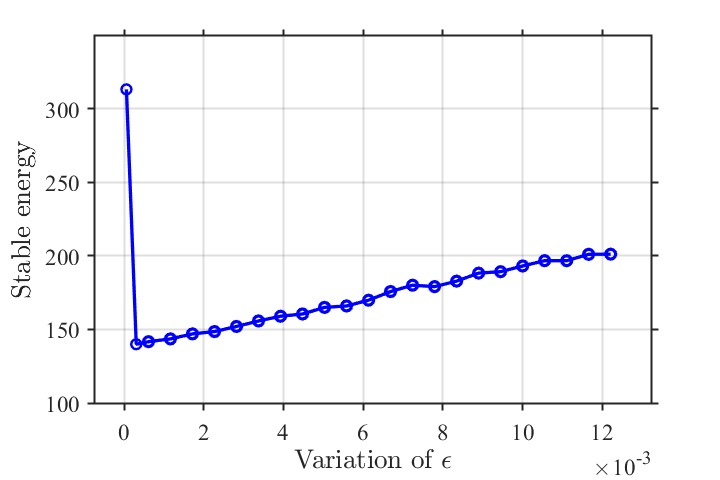}
    \label{fig:Energy_sigularity_siam}
  }
  \caption{(a) Evolution of the singularity radius (b) the corresponding minimized energy as the regularization parameter $\varepsilon = \tau$ varies. See Section~\ref{sec:influence}.}
  \label{fig:Radiusandsingularity}
\end{figure}

The dependence on $\alpha$ is further quantified in~\Cref{fig:Radiusandsingularity}. The curves in (a) and (b) report, respectively, the change of steady-state singularity radius and the Dirichlet energy on $\Omega_1$ as $\alpha$ is refined from $1$ to $1/150$. Over this range, the steady-state energy decreases monotonically as the singularities move outward, consistent with the mechanism described in Section~\ref{sec:Ginzburg}, where minimizers are governed by the renormalized energy.

When $\alpha$ is reduced below an effective resolution threshold (approximately $1/150h$), a qualitative change is observed. Although the iteration remains stable and continues to produce four singularities, their radii shift abruptly to $r\approx 0.2$. This behavior indicates that, for sufficiently small $\alpha$, discretization effects become dominant and the computed steady state no longer reflects the intended sharp-interface limit. In the reported experiments, the method exhibits robust and physically consistent behavior for $\alpha\in(1/64,1)$.

\subsection{Performance under high curvature and sharp corners} \label{sec: perf}

Sections~\ref{sec:Ginzburg} and~\ref{sec:UnconditionalDecrease} focus on how the variational model~\eqref{eq:originproblem} selects singularity configurations. From a meshing perspective, however, the relevance of Dirichelt energy as a proxy for mesh quality is not immediate: as noted in~\cite{Osting2019}, lower-energy configurations do not necessarily yield better block alignment or higher-quality quadrilateral elements. In addition, many practical geometries contain $C^0$ regularity boundaries (e.g., polygons with corners), whereas the effect of reduced boundary regularity on cross-field models is less well understood.

This subsection investigates the behavior of the proposed methods on domains with concentrated curvature and with sharp corners. Results are compared against solutions of~\eqref{eq:originproblem} computed by the MBO method, with the goal of assessing robustness with respect to geometric irregularities and the extent to which the parameters provide control over the resulting singularity pattern and separatrix-induced decomposition.

\begin{figure}[t!]
  \centering

  \includegraphics[width=0.3\textwidth]{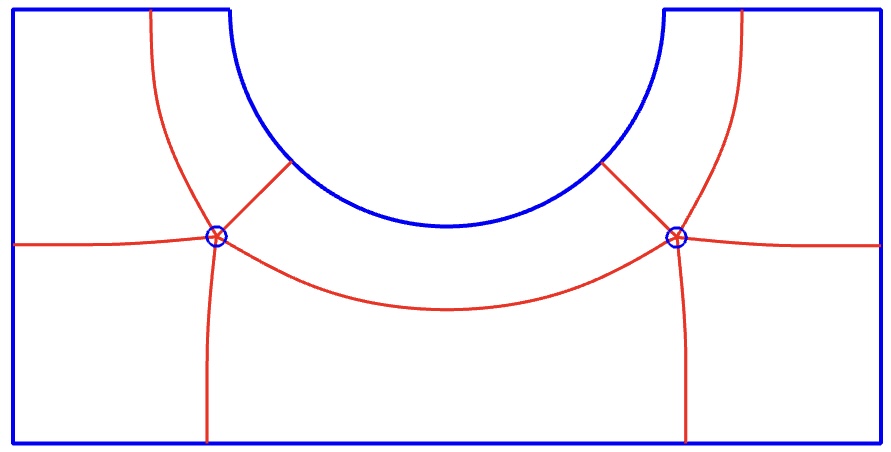}
  \hspace{2mm}
  \includegraphics[width=0.3\textwidth]{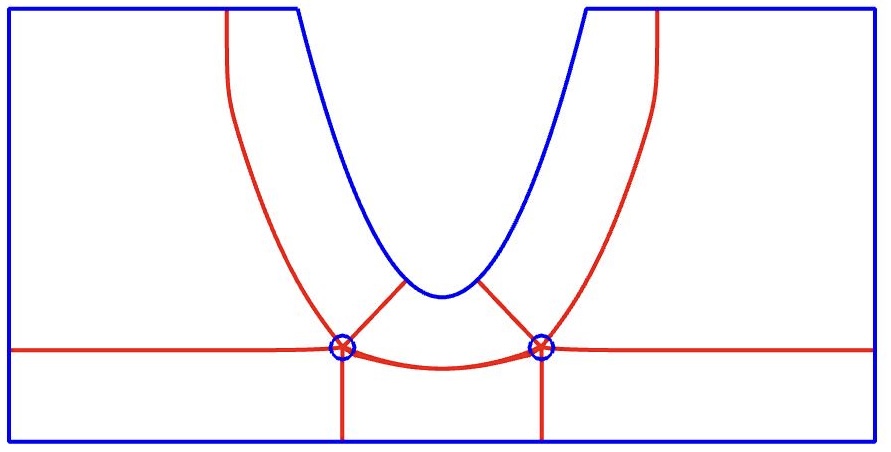}
  \hspace{2mm}
  \includegraphics[width=0.3\textwidth]{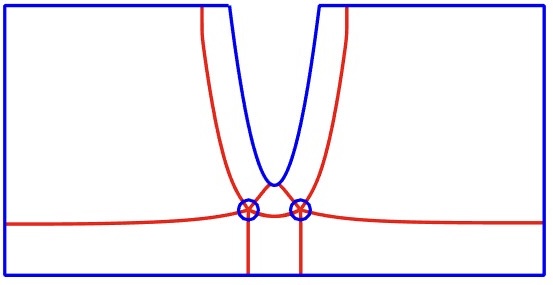}
  
  \vspace{3mm}  
  
  \includegraphics[width=0.3\textwidth]{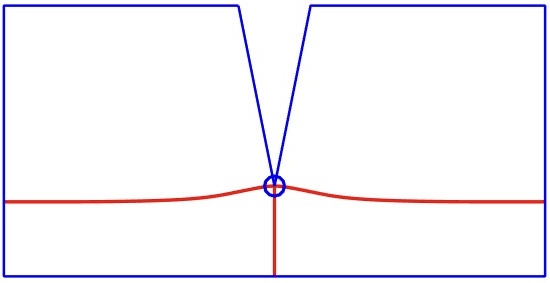}
  \hspace{2mm}
  \includegraphics[width=0.3\textwidth, height=0.15\textwidth]{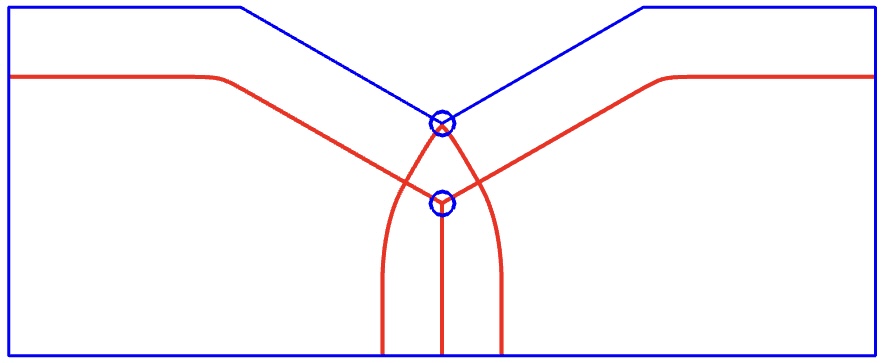}
  \hspace{2mm}
  \includegraphics[width=0.3\textwidth]{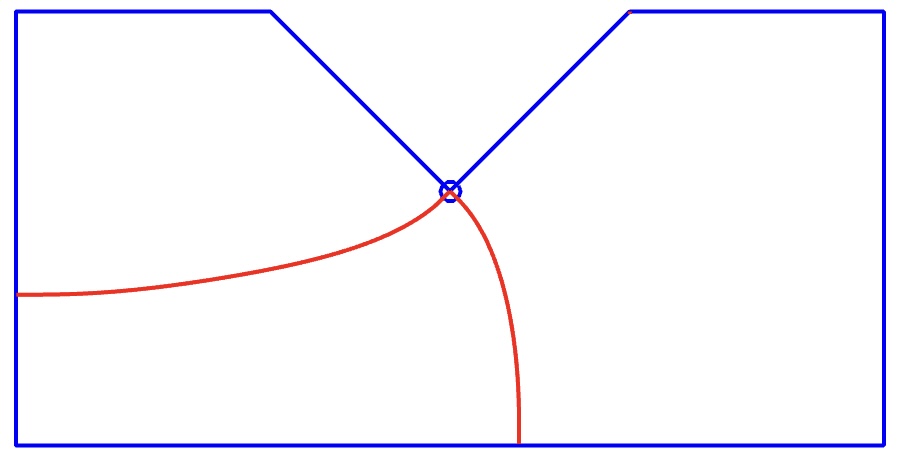}
  
  \vspace{3mm}  
  
  \includegraphics[width=0.3\textwidth]{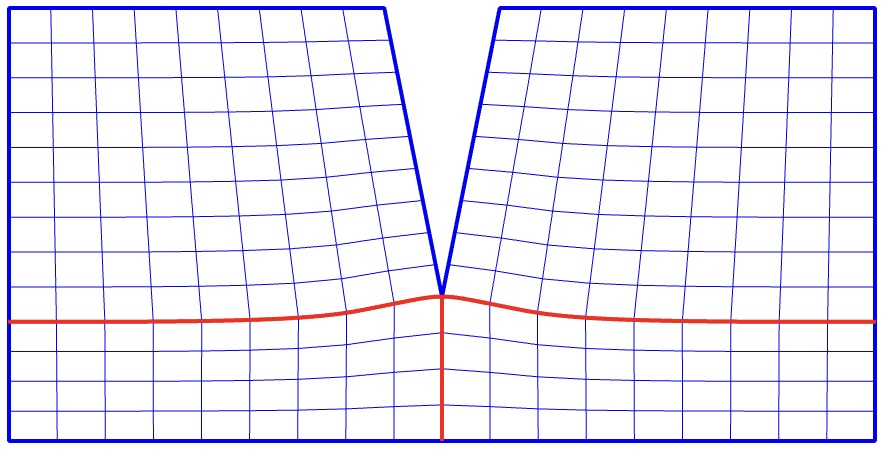}
  \hspace{2mm}
  \includegraphics[width=0.3\textwidth, height=0.15\textwidth]{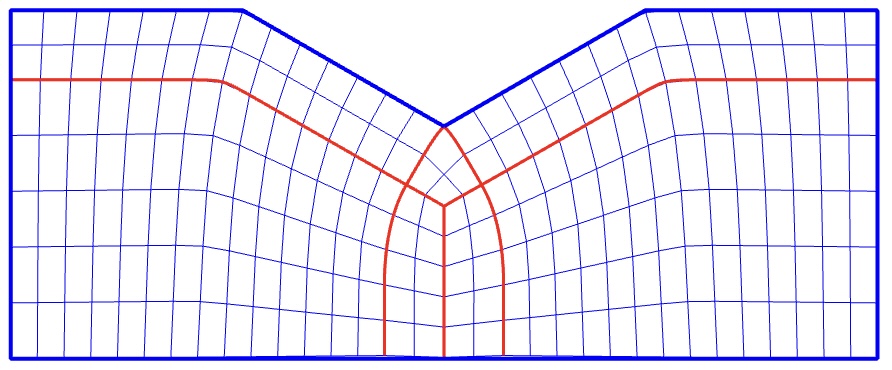}
  \hspace{2mm}
  \includegraphics[width=0.3\textwidth]{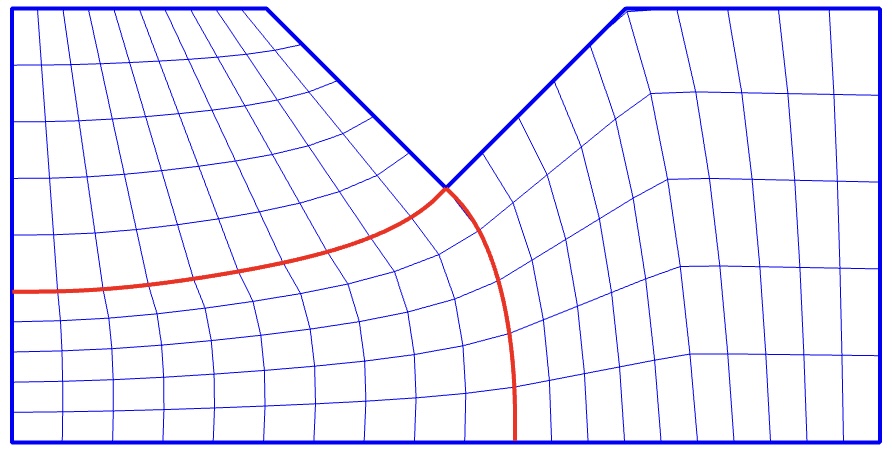}

  \caption{Singularities and streamlines generated by \Cref{alg:iterative_projection} under different boundary curvatures. The first and second rows correspond to different curvature settings: $y = \sqrt{1-x^2}$, $y = 4x^2$, $y = 16x^2$, $y = 4|x|$, $y = \frac{1}{\sqrt{3}}|x|$ and $y = |x|$, respectively. See Section~\ref{sec: perf}.}
  \label{fig:C0boundary}
\end{figure}

\begin{figure}[t!]
  \centering
  \includegraphics[width=0.4\textwidth]{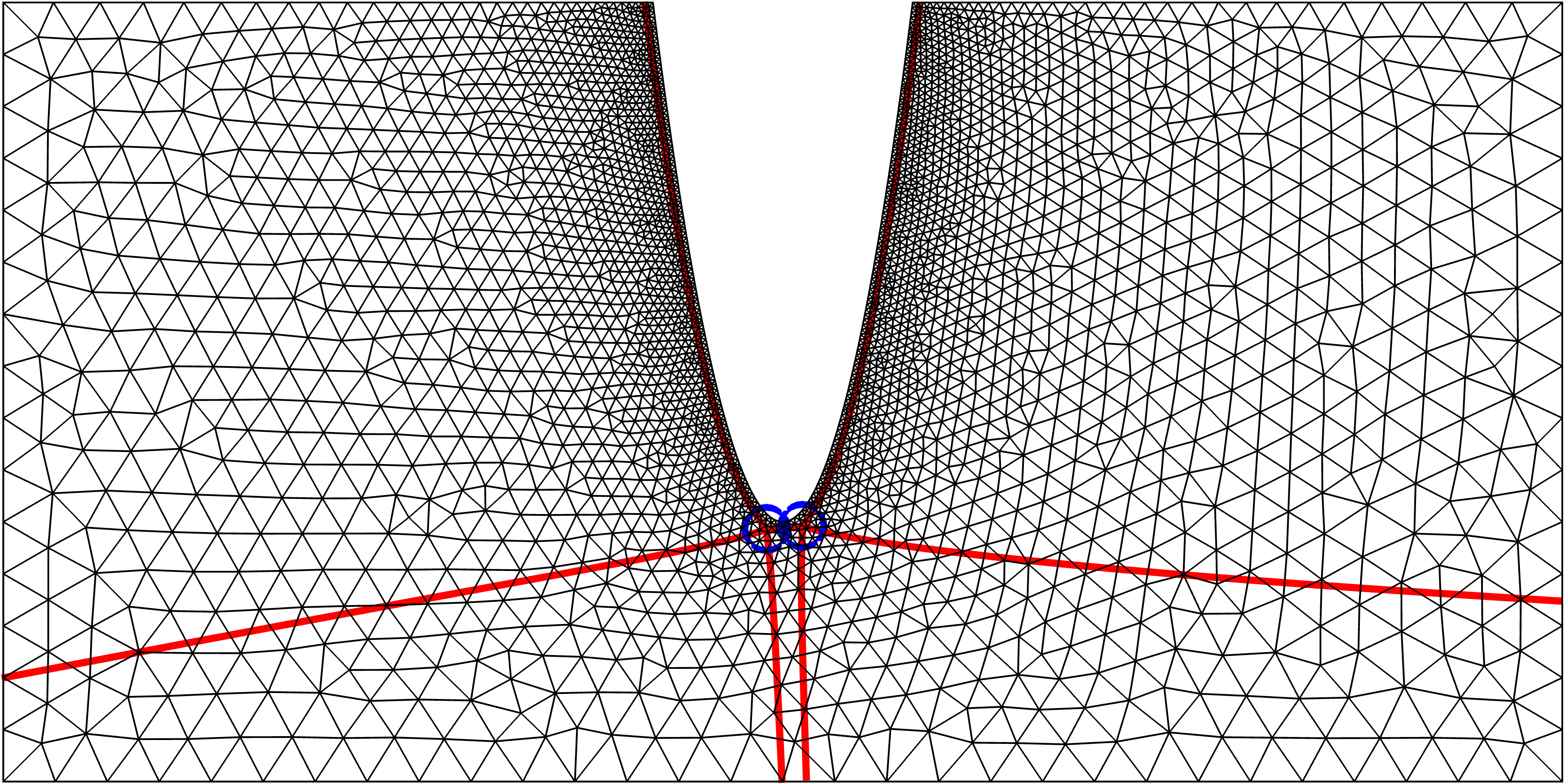}
  \hspace{1cm}
  \includegraphics[width=0.25\textwidth]{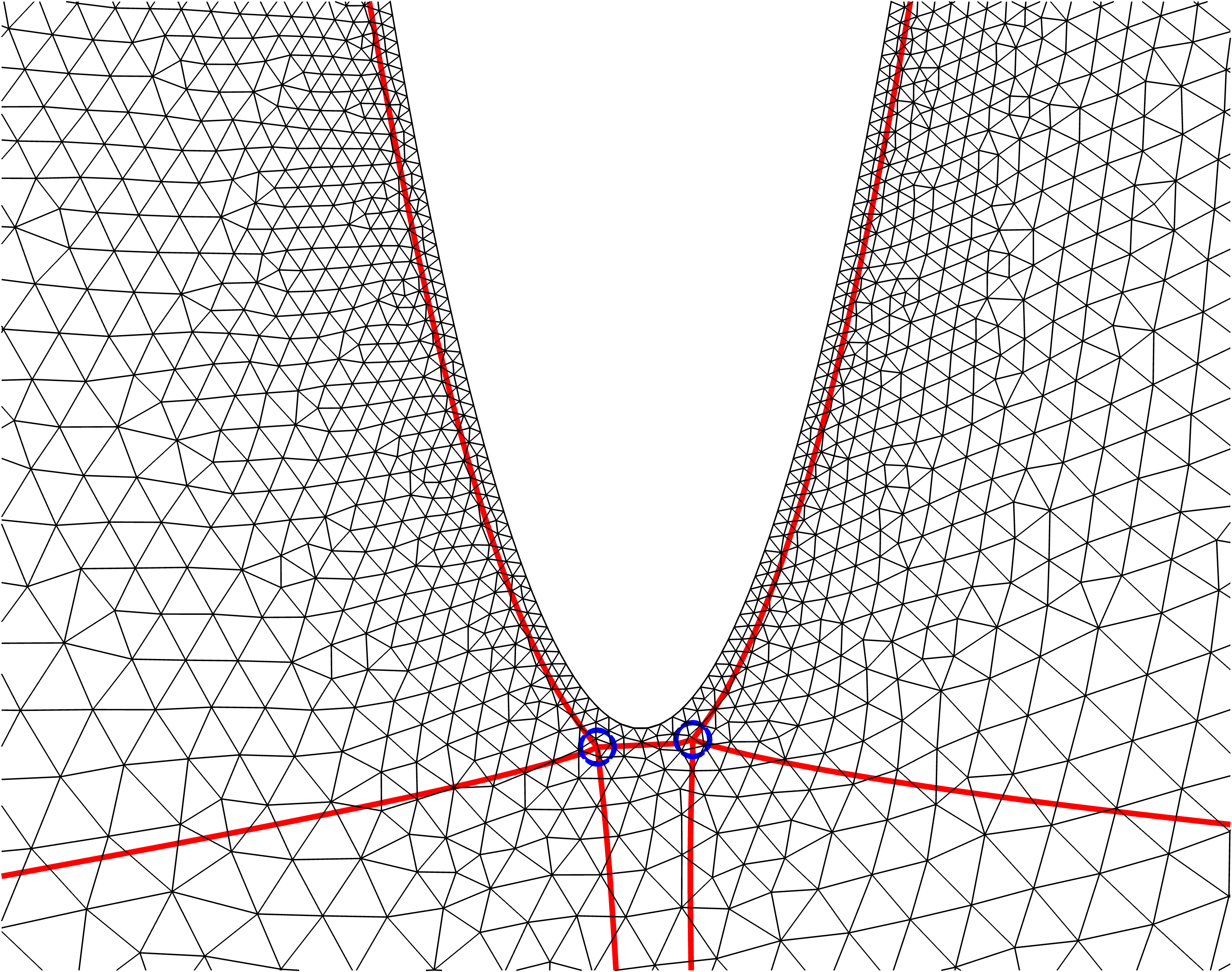}
 \caption{Left: Streamline partitions generated from the cross field corresponds to  $y = 16x^2$, computed using the MBO method. Right: Zoom out of the region around the singularities in the left figure.  See Section~\ref{sec: perf}.}
  \label{fig:MBO_H3}
\end{figure}

\begin{figure}[t!]
    \centering

    \begin{subfigure}[t]{0.18\textwidth}
        \centering
        \includegraphics[width=\linewidth]{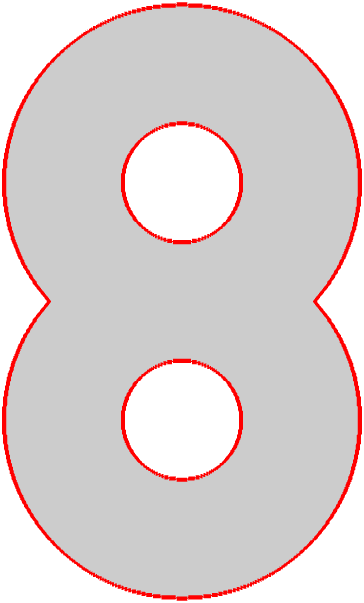}
        \label{fig:8_indicator}
    \end{subfigure}
    \hspace{22mm}
    \begin{subfigure}[t]{0.18\textwidth}
        \centering
        \includegraphics[width=\linewidth]{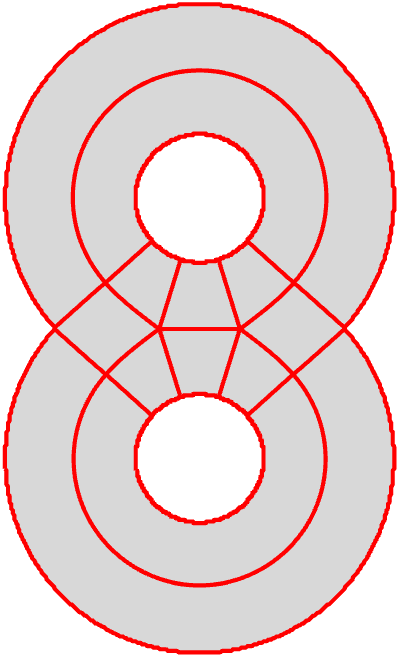}
        \label{fig:8_streamline}
    \end{subfigure}
    \hspace{22mm}
    \begin{subfigure}[t]{0.18\textwidth}
        \centering
        \includegraphics[width=\linewidth]{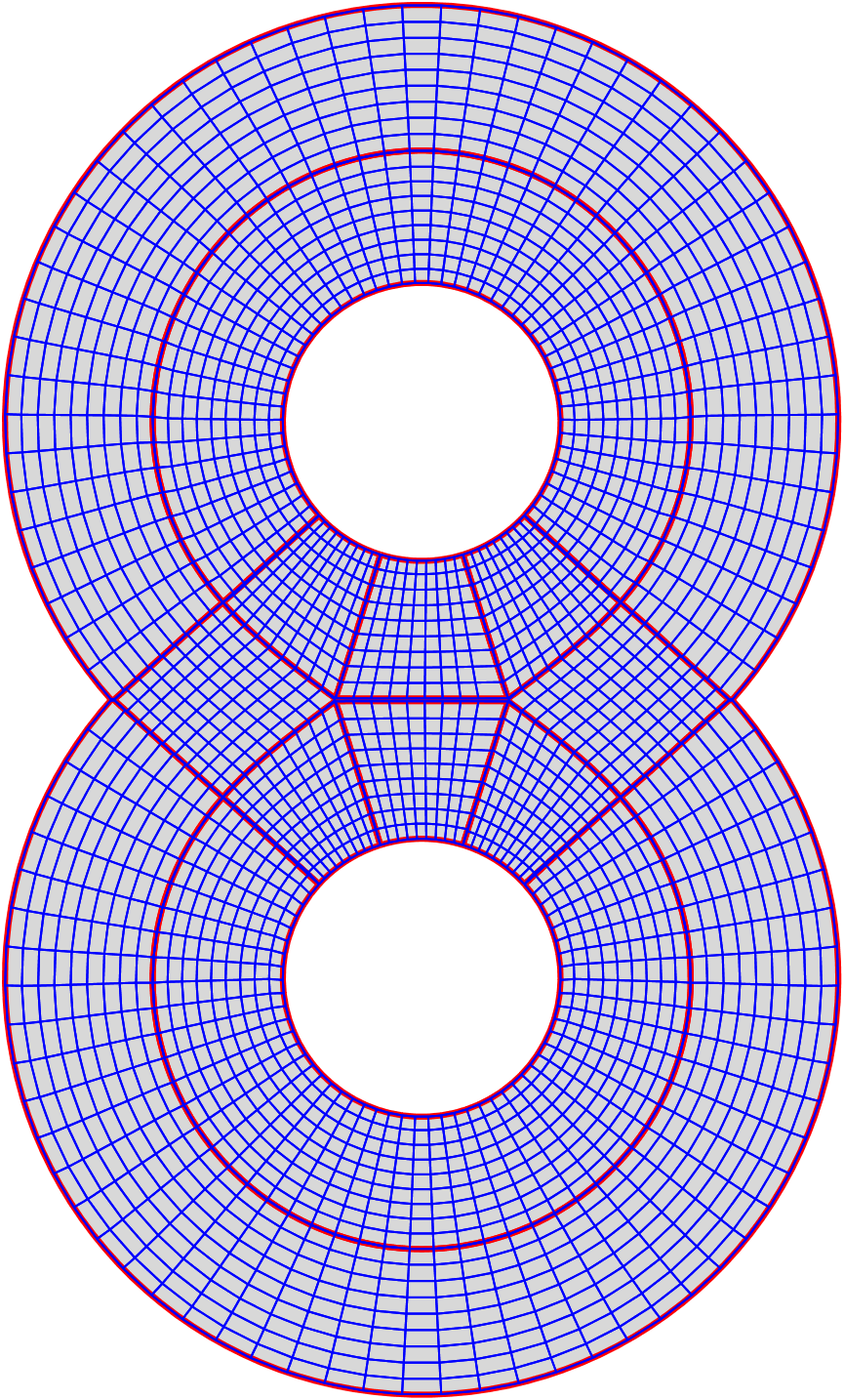}
        \label{fig:8_mehsgen}
    \end{subfigure}

    \begin{subfigure}[t]{0.27\textwidth}
        \centering
        \includegraphics[width=\linewidth]{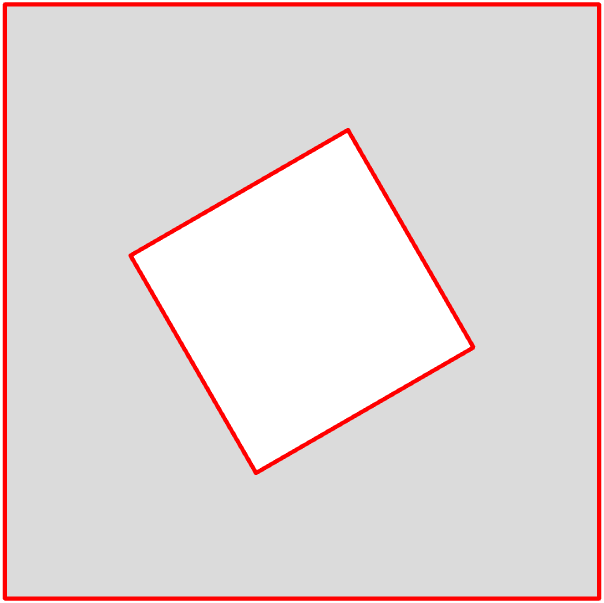}
        \label{fig:H6_indicator}
    \end{subfigure}
    \hspace{8mm}
    \begin{subfigure}[t]{0.27\textwidth}
        \centering
        \includegraphics[width=\linewidth]{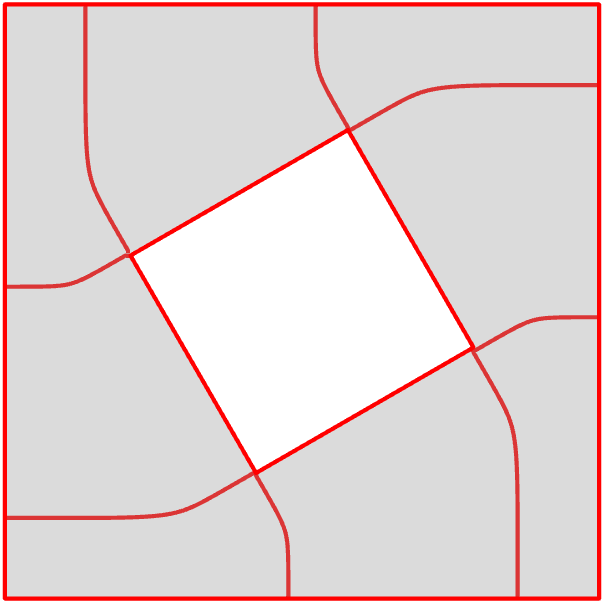}
        \label{fig:H6_streamline}
    \end{subfigure}
    \hspace{8mm}
    \begin{subfigure}[t]{0.27\textwidth}
        \centering
        \includegraphics[width=\linewidth]{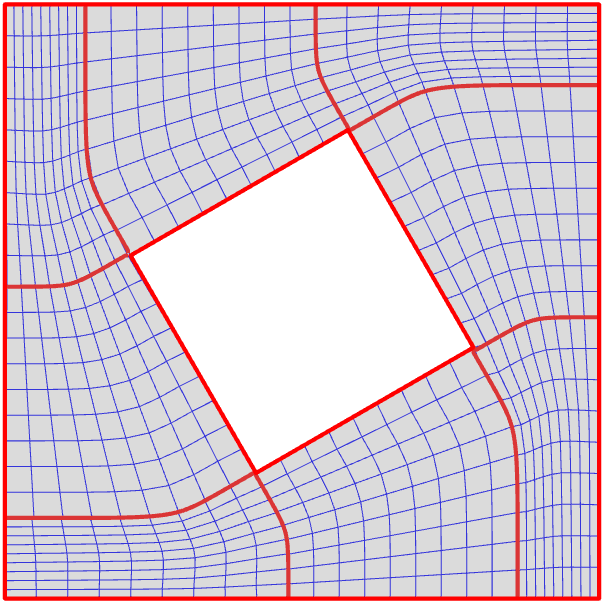}
        \label{fig:H6_meshgen}
    \end{subfigure}

    \begin{subfigure}[t]{0.3\textwidth}
        \centering
        \includegraphics[width=\linewidth]{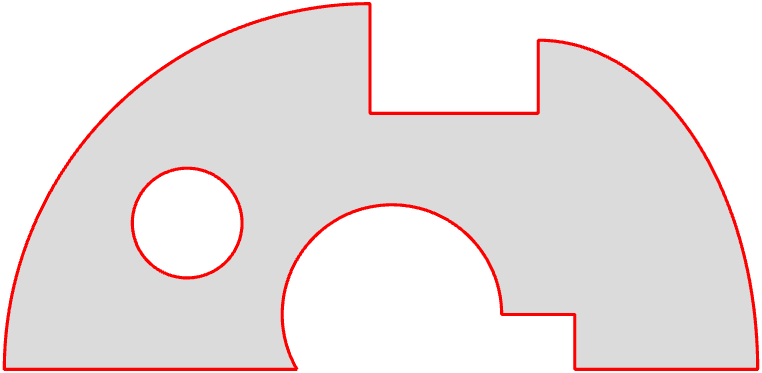}
        \label{fig:E1_indicator}
    \end{subfigure}
    \hspace{2mm}
    \begin{subfigure}[t]{0.3\textwidth}
        \centering
        \includegraphics[width=\linewidth]{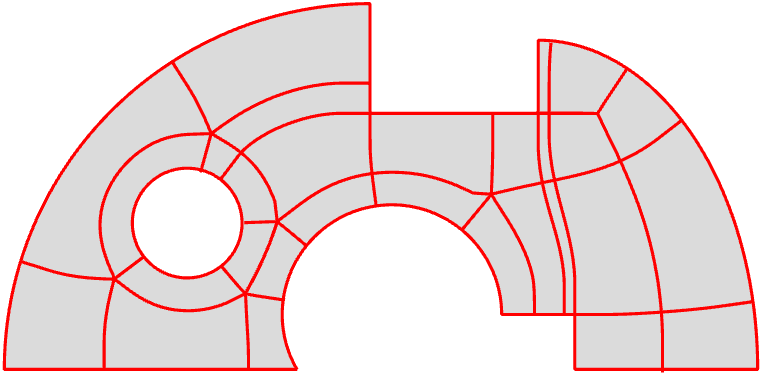}
        \label{fig:E1_streamline}
    \end{subfigure}
    \hspace{2mm}
    \begin{subfigure}[t]{0.3\textwidth}
        \centering
        \includegraphics[width=\linewidth]{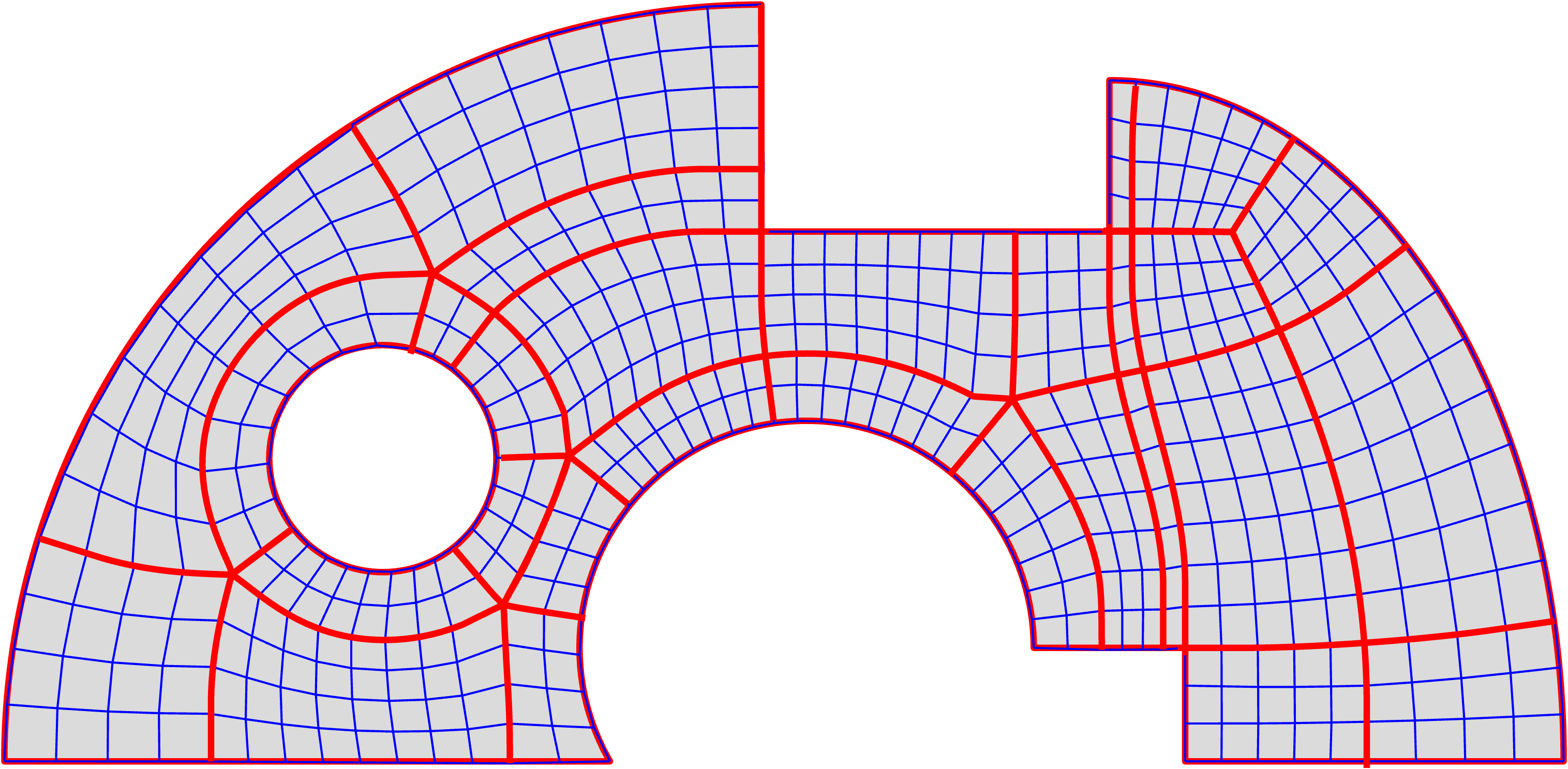}
        \label{fig:E1_meshgen}
    \end{subfigure}

    \begin{subfigure}[t]{0.3\textwidth}
        \centering
        \includegraphics[width=\linewidth]{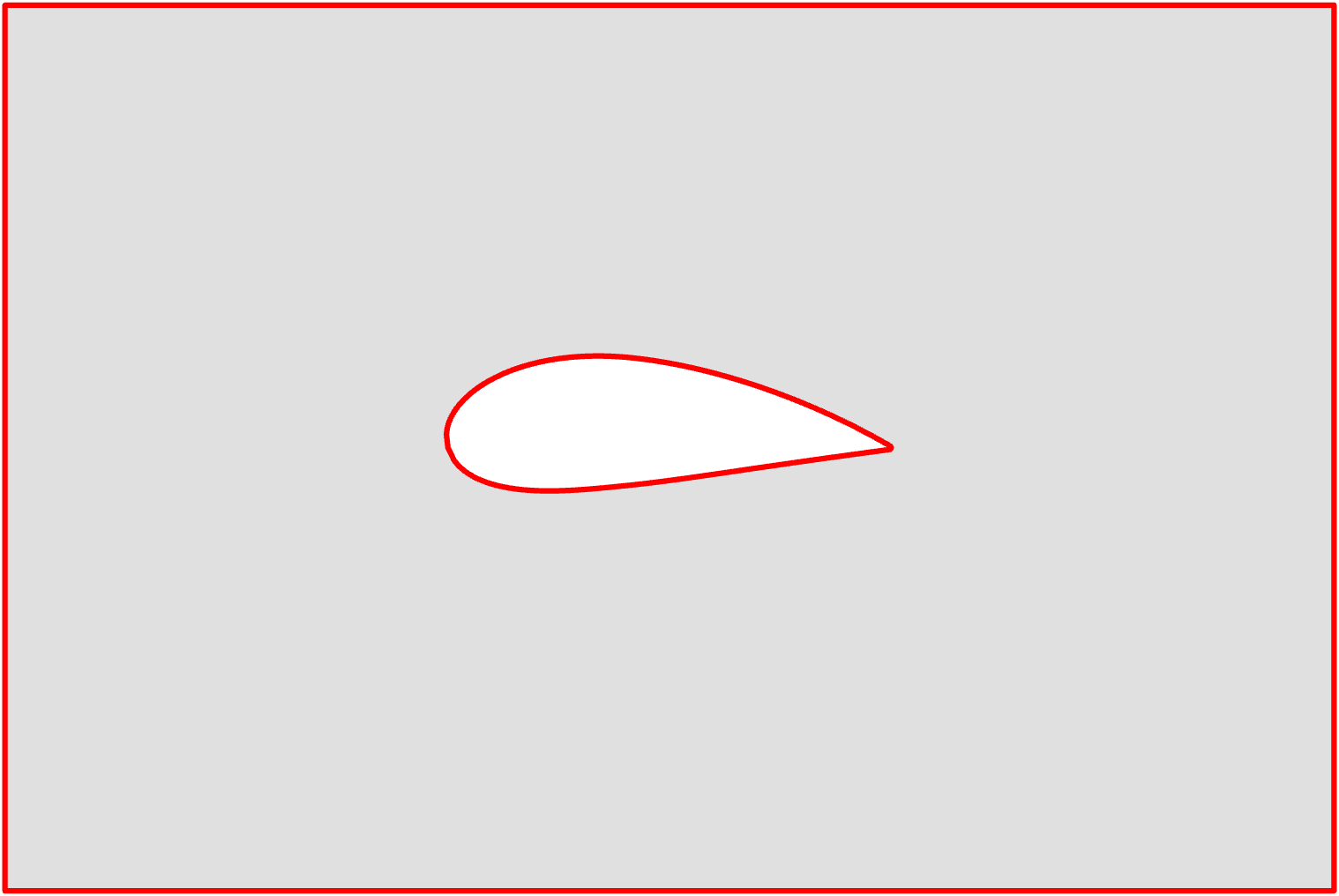}
        \label{fig:E2_indicator}
    \end{subfigure}
    \hspace{2mm}
    \begin{subfigure}[t]{0.3\textwidth}
        \centering
        \includegraphics[width=\linewidth]{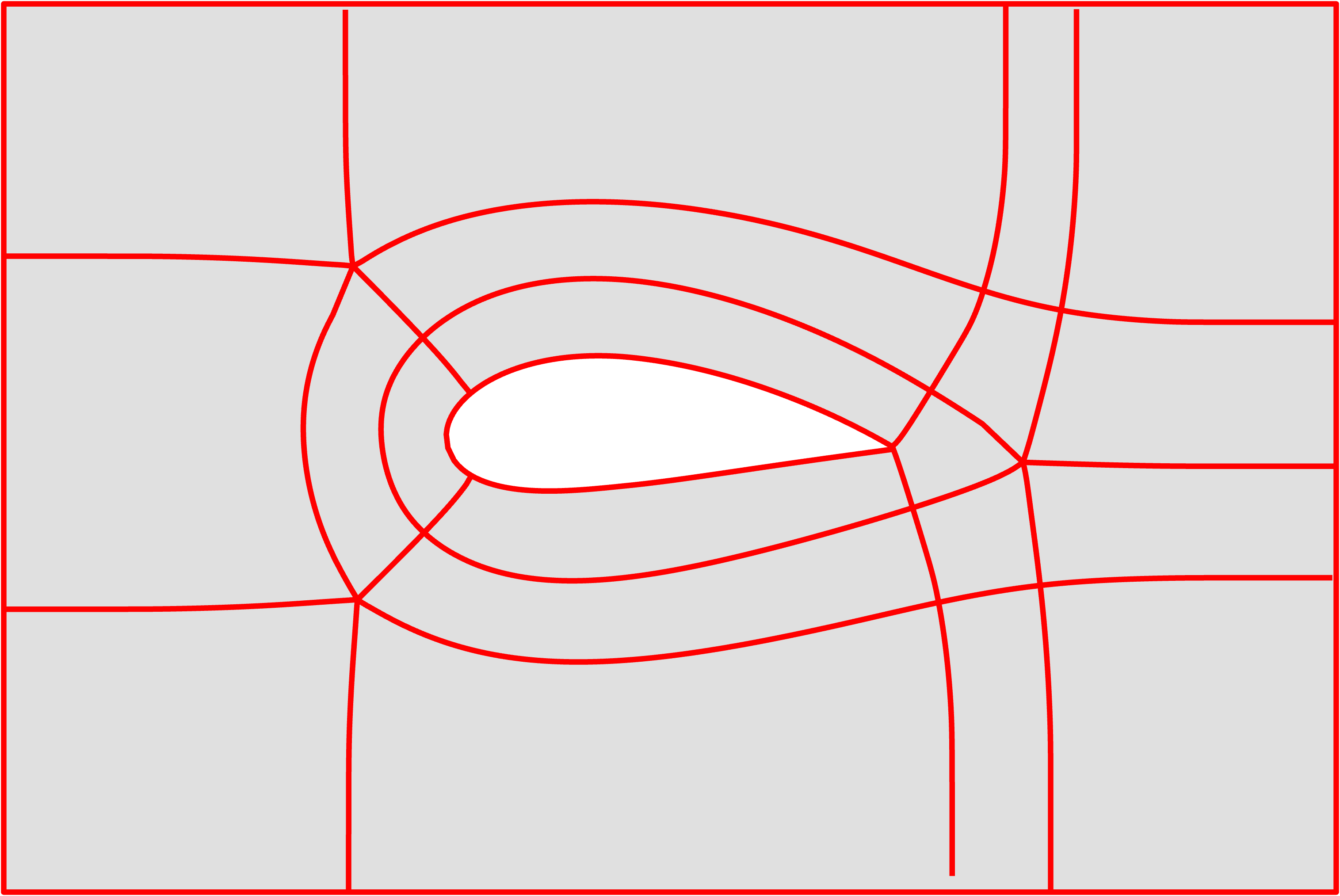}
        \label{fig:E2_streamline}
    \end{subfigure}
    \hspace{2mm}
    \begin{subfigure}[t]{0.3\textwidth}
        \centering
        \includegraphics[width=\linewidth]{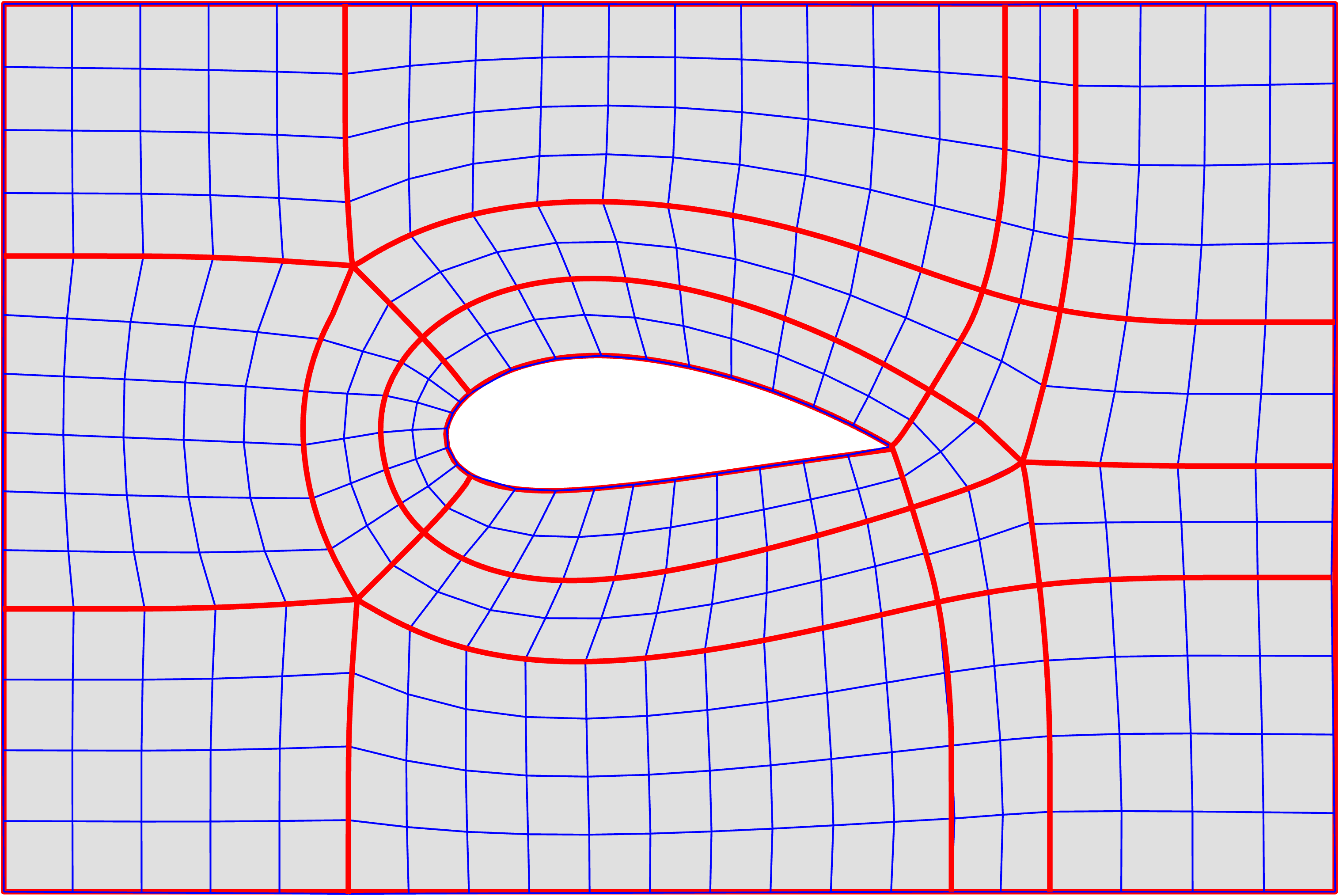}
        \label{fig:E2_meshgen}
    \end{subfigure}

    \caption{Quad mesh generation results for various geometries. Each row displays: (Left) the input domain; (Center) the topological decomposition induced by the cross field separatrices; and (Right) the resulting quad mesh with the decomposition structure overlaid in red. See Section~\ref{sec:exMG}.}
    \label{fig:Example_quad_mesh}
\end{figure}

To isolate the effect of boundary curvature on the singularity configuration, we consider a modified rectangular domain
\[
\Omega_1 = [-1.5,1.5]\times[-0.5,1],
\]
in which the top boundary segment is replaced by a curve that bends into the interior. We first study three smooth profiles,
\[
y=-\sqrt{1-x^2},\qquad y=4x^2,\qquad y=16x^2,
\]
which correspond to increasingly concentrated curvature near the origin. We then examine three geometries with $C^0$ boundary singularities, in which the interior angle at the origin varies from acute to obtuse, including the right-angle case $\pi/2$.

The resulting cross fields and meshes are shown in~\cref{fig:C0boundary}. In the smooth cases (first row), increasing curvature concentration draws the two interior singularities of index $-\tfrac14$ progressively closer to the top boundary. When the boundary becomes non-smooth (second row), these two defects merge at the corner, producing an effective boundary singularity of index $-\tfrac12$. This transition is not captured by standard Ginzburg--Landau asymptotics, indicating that the limiting process from smooth to $C^0$ boundaries can violate the regularity assumptions under which the renormalized energy is typically derived and interpreted. As the corner angle increases, further changes occur: in the right-angle case, a boundary singularity persists (with index $+\tfrac14$ in our experiments), whereas for an obtuse corner an additional interior singularity of index $+\tfrac14$ appears. Overall, these experiments demonstrate a nontrivial dependence of the singularity pattern on geometric singularities and boundary regularity.

The comparison with MBO computations in~\Cref{fig:MBO_H3} also highlights practical limitations of solving~\eqref{eq:originproblem} by MBO on domains with strongly varying curvature. Rapid changes in boundary normals impose stringent spatial resolution requirements; maintaining accuracy typically calls for adaptive, highly nonuniform meshes. In addition, the clustering of singularities near high-curvature regions may induce irregular and unbalanced block decompositions, which can deteriorate element quality and complicate control of mesh-size transitions across block interfaces. In contrast, the proposed method operates on a fixed background grid and remains stable under the same geometric variations; moreover, through the parameters $(\epsilon,\tau)$, it provides a mechanism to influence the equilibrium placement of singularities, which is advantageous when constructing block decompositions in complex geometries.

\subsection{Example mesh generation} \label{sec:exMG}
In the preceding sections, we investigated the stability, efficiency, and tunability of the proposed algorithm across a variety of settings. These studies demonstrate that the method remains robust under varying parameter choices and geometric configurations. In this final example, we apply the full meshing framework to several representative domains, including composite and topologically nontrivial geometries. This experiment serves to validate the end-to-end capability of the algorithm for practical mesh generation tasks. The results, presented in \Cref{fig:Example_quad_mesh}, illustrate the adaptability of the method to complex regions, highlighting its applicability to general computational domains.

\section{Conclusions}
\label{sec:conclusions}
In this paper, an efficient and stable method is proposed for generating block-structured quadrilateral meshes via energy minimization of a relaxed Ginzburg--Landau model. By extending the problem to a regular computational domain and applying Laplacian of the exponential, the proposed approach enables fast and stable computation using FFT-based diffusion and projection steps. The proposed framework avoids the need for input triangulations or complex solvers, while ensuring monotonic energy decreasing. Moreover, it allows flexible control over singularity positions by exploring multiple energy levels, addressing limitations in existing techniques. The numerical results confirm the efficiency, robustness, and ability of the method to produce high-quality meshes with customizable singularity layouts. 

\section*{Acknowledgments}
We would like to thank Prof. Na Lei (Dalian University of Technology) and Prof. Falai Chen (University of Science and Technology of China) for their valuable suggestions and insightful discussions.

\bibliographystyle{siamplain}
\bibliography{references}

\end{document}